\documentclass{article}
\usepackage{times,authblk,enumerate}
\usepackage{amsmath, amsfonts, amssymb, amsthm, bm, stmaryrd, graphicx}
\usepackage[font=small,labelfont=bf]{caption}
\usepackage{tikz,ifthen}
\usepackage[shortcuts]{extdash}
\usetikzlibrary{patterns}

\renewcommand{\epsilon}{\varepsilon}
\renewcommand{\theta}{\vartheta}
\renewcommand{\phi}{\varphi}
\renewcommand{\rho}{\varrho}

\newtheorem{Def}{Definition}[section]
\newenvironment{definition}{\begin{Def} \rm}{\end{Def}}
\newtheorem{lemma}[Def]{Lemma}
\newtheorem{proposition}[Def]{Proposition}

\newtheorem{theorem}[Def]{Theorem}

\newtheorem{remark}[Def]{Remark}

\newcommand{\axiom}[1]{\item[\rm (#1)]}

\newcommand{\sep}{;}

\newcommand{\Naturals}{{\mathbb N}}

\renewcommand{\setminus}{\mathbin{\raisebox{0.2ex}{$\smallsetminus$}}}
\newcommand{\disjointunion}{\mathbin{\dot\cup}}
\newcommand{\binaryoperation}{\mathbin{\diamond}} 

\newcommand{\So}{S}                              
\newcommand{\withoutzero}[1]{#1^\star}           
\newcommand{\oneone}{(1,\!1)}
\renewcommand{\leq}{\leqslant}                   
\renewcommand{\geq}{\geqslant}                   
\newcommand{\congruence}{\mathbin{\approx}   }   
\newcommand{\congrcl}[1]{{\langle #1 \rangle_\approx}} 
\newcommand{\rees}{\mathbin{\approx}}            
\newcommand{\Se}{\bar{S}}                        
\newcommand{\zeroo}{0}                           
\newcommand{\zeroe}{0}                           
\newcommand{\atom}{\alpha}                       
\newcommand{\atome}{\alpha}                      
\newcommand{\mo}{\mathbin{\odot}}                
\newcommand{\me}{\mathbin{\bar\odot}}            
\newcommand{\support}{\mathcal P}                
\newcommand{\cosupport}{\mathcal Q}              
\newcommand{\eqo}{\mathbin{\sim}}                
\newcommand{\noteqo}{\mathbin{\not\sim}}         
\newcommand{\eqp}{\mathbin{\dot{\sim}}}          
\newcommand{\noteqp}{\mathbin{\dot{\nsim}}}      
\newcommand{\eqe}{\mathbin{\bar{\sim}}}          
\newcommand{\idp}{\epsilon}                      
\newcommand{\leftidp}{\idp_l}
\newcommand{\rightidp}{\idp_r}
\newcommand{\leqcomp}{\mathbin{\trianglelefteqslant}}  
\newcommand{\geqcomp}{\mathbin{\trianglerighteqslant}}
\newcommand{\equ}{\mathbin{\sim}}                 
\newcommand{\equcl}[1]{{\langle #1 \rangle_\sim}} 
\newcommand{\eqpcl}[1]{{\langle #1 \rangle_{\dot{\sim}}}}
\newcommand{\eqecl}[1]{{\langle #1 \rangle_{\bar{\sim}}}}
\newcommand{\preordersim}{\mathbin{\leq_\sim}}
\newcommand{\preorderblock}{\mathbin{\trianglelefteqslant_\sim}}
\newcommand{\preorderblockp}{\mathbin{\trianglelefteqslant_{\dot{\sim}}}}
\newcommand{\preorder}{\preccurlyeq}
\newcommand{\preorderequ}{\approx}
\newcommand{\preorderequcl}[1]{{\langle #1 \rangle_\approx}}

\newcommand{\fntomonoid}{f.n.\ tomonoid}

\tikzset{rectangle in table/.style={color=black,ultra thick}}

\edef\reccircsize{0.4}

\tikzset{node text/.style={text height=1.5ex,text depth=0.25ex}}

\tikzset{tom table line/.style={style=thin}}

\tikzset{lvl 0/.style={color=black}}
\tikzset{lvl t/.style={color=black}}
\tikzset{lvl u/.style={color=black}}
\tikzset{lvl v/.style={color=black}}
\tikzset{lvl w/.style={color=black}}
\tikzset{lvl x/.style={color=black}}
\tikzset{lvl y/.style={color=black}}
\tikzset{lvl z/.style={color=black}}
\tikzset{lvl 1/.style={color=black}}

\tikzset{lvl fill 0/.style={color=gray, fill=white}}
\tikzset{lvl fill t/.style={color=black, fill=lightgray}}
\tikzset{lvl fill u/.style={color=black, fill=lightgray}}
\tikzset{lvl fill v/.style={color=black, fill=lightgray}}
\tikzset{lvl fill w/.style={color=black, fill=lightgray}}
\tikzset{lvl fill x/.style={color=black, fill=lightgray}}
\tikzset{lvl fill y/.style={color=black, fill=lightgray}}
\tikzset{lvl fill z/.style={color=black, fill=lightgray}}
\tikzset{lvl fill 1/.style={color=black, fill=lightgray}}

\tikzset{lvl a/.style={color=black}}
\tikzset{lvl fill a/.style={color=black, fill=white}}

\tikzset{lvl ?/.style={color=black}}
\tikzset{lvl fill ?/.style={color=black, fill=white}}

\tikzset{lvl !/.style={color=black}}
\tikzset{lvl fill !/.style={color=black, fill=white}}

\tikzset{rectangle line/.style={color=gray,ultra thick, dashed}}
\tikzset{rectangle circle/.style={color=black,ultra thick}}
\tikzset{levelset antecedent/.style={ultra thick}}
\tikzset{levelset consequent/.style={dotted, ultra thick}}

\def\tomcellwrite[#1](#2,#3,#4){
    \begin{scope}
        \pgfmathparse{#2-0.5}
        \edef\xx{\pgfmathresult}
        \pgfmathparse{#3-0.5}
        \edef\yy{\pgfmathresult}
        \node[node text,#1] at (\xx,\yy) {#4};
    \end{scope}
}

\def\tomcellwriteright[#1](#2,#3,#4){
    \begin{scope}
        \edef\xx{#2-1}
        \pgfmathparse{#3-0.5}
        \edef\yy{\pgfmathresult}
        \node[node text, right, #1] at (\xx,\yy) {#4};
    \end{scope}
}

\def\tomcellwriteleft[#1](#2,#3,#4){
    \begin{scope}
        \pgfmathparse{#2}
        \edef\xx{\pgfmathresult}
        \pgfmathparse{#3-0.5}
        \edef\yy{\pgfmathresult}
        \node[node text, left, #1] at (\xx,\yy) {#4};
    \end{scope}
}

\def\tomcellfill[#1](#2,#3,#4,#5){
    \begin{scope}
        \pgfmathparse{#2}
        \edef\xa{\pgfmathresult}
        \pgfmathparse{#3-1}
        \edef\xb{\pgfmathresult}
        \pgfmathparse{#4}
        \edef\ya{\pgfmathresult}
        \pgfmathparse{#5-1}
        \edef\yb{\pgfmathresult}
        \draw[#1] (\xa,\ya) rectangle (\xb,\yb);
    \end{scope}
}

\newcommand{\tomunit}[1]{
    \begin{scope}
        \pgfmathparse{-#1}
        \edef\zax{\pgfmathresult}
        \edef\ax{1}
        \tomcellwrite[](0,\ax,$1$);
        \tomcellwrite[](\ax,0,$1$);
        \tomcellwrite[](0,0,$1$);
    \end{scope}
}

\newcommand{\tomzero}[1]{
    \begin{scope}
        \pgfmathparse{-#1}
        \edef\zax{\pgfmathresult}
        \edef\ax{1}
        \pgfmathparse{-#1+1}
        \edef\ze{\pgfmathresult}
        \tomcellwrite[](\ze,\ax,$0$);
        \tomcellwrite[](\ax,\ze,$0$);
        \tomcellwrite[](\ze,0,$0$);
        \tomcellwrite[](0,\ze,$0$);
    \end{scope}
}

\newcommand{\tomtable}[1]{
    \begin{scope}
        \foreach \a in {0, 1, ..., #1}{
            \draw[tom table line] (0,-\a) -- (-#1,-\a);
            \draw[tom table line] (-\a,0) -- (-\a,-#1);
        }
    \end{scope}
}

\newcommand{\tomtablefilled}[1]{
    \begin{scope}
        \foreach \tomrow [count=\tomsize] in #1 {}
        \foreach \tomrow [count=\i] in #1 {
            \begin{scope}
                \pgfmathparse{1-\i}
                \edef\y{\pgfmathresult}
                \foreach \val [count=\j] in \tomrow {
                    \pgfmathparse{\j-\tomsize}
                    \edef\x{\pgfmathresult}
                    \tomcellfill[lvl fill \val](\x,\x,\y,\y);
                    \ifthenelse{\equal{\val}{a}}{
                        \tomcellwrite[lvl \val](\x,\y,{$\atom$});
                    }{
                        \ifthenelse{\equal{\val}{?}}{
                        }{
                            \ifthenelse{\equal{\val}{!}}{
                                \tomcellwrite[lvl \val](\x,\y,{$\small\textbf{!}$});
                            }{
                                \tomcellwrite[lvl \val](\x,\y,{$\val$});
                            }
                        }
                    }
                }
            \end{scope}
        }
        \tomtable{\tomsize}
    \end{scope}
}

\newcommand{\tomaxis}[1]{
    \begin{scope}
        \foreach \tomrow [count=\tomsize] in #1 {}
        \foreach \tomrow [count=\i] in #1 {
            \begin{scope}
                \pgfmathparse{1-\i}
                \edef\y{\pgfmathresult}
                \foreach \val [count=\j] in \tomrow {
                    \pgfmathparse{\j-\tomsize}
                    \edef\x{\pgfmathresult}
                    \edef\ax{1}
                    \ifthenelse{\equal{\val}{a}}{
                        \tomcellwrite[lvl \val](\x,\ax,{$\atom$});
                        \tomcellwrite[lvl \val](\ax,\x,{$\atom$});
                    }{
                        \tomcellwrite[lvl \val](\x,\ax,{$\val$});
                        \tomcellwrite[lvl \val](\ax,\x,{$\val$});
                    }
                }
                \breakforeach
            \end{scope}
        }
    \end{scope}
}

\newcommand{\tomsupp}[1]{
    \begin{scope}
        \foreach \tomrow [count=\tomsize] in #1 {}
        \foreach \tomrow [count=\i] in #1 {
            \begin{scope}
                \pgfmathparse{1-\i}
                \edef\y{\pgfmathresult}
                \foreach \val [count=\j] in \tomrow {
                    \pgfmathparse{\j-\tomsize}
                    \edef\x{\pgfmathresult}
                    \ifthenelse{\equal{\val}{0}}{
                        \tomcellfill[color=white, fill=white](\x,\x,\y,\y);
                    }{
                        \ifthenelse{\equal{\val}{a}}{
                            \tomcellfill[color=white, fill=white](\x,\x,\y,\y);
                            \tomcellwrite[color=black, fill=white](\x,\y,$\atom$);
                        }{
                            \tomcellfill[color=lightgray, fill=lightgray](\x,\x,\y,\y);
                        }
                    }
                }
            \end{scope}
        }
    \end{scope}
}

\def\tomrectangleside[#1](#2,#3,#4,#5){
    \begin{scope}
        \edef\rr{\reccircsize}
        \edef\xa{#2}
        \edef\ya{#3}
        \edef\xb{#4}
        \edef\yb{#5}
        \pgfmathparse{(\xa+\xb)/2}
        \edef\xc{\pgfmathresult}
        \pgfmathparse{(\ya+\yb)/2}
        \edef\yc{\pgfmathresult}
        \pgfmathparse{(abs(\xa-\xb)/2)-\rr}
        \edef\lenhori{\pgfmathresult}
        \pgfmathparse{(abs(\ya-\yb)/2)-\rr}
        \edef\lenvert{\pgfmathresult}
        \pgfmathparse{veclen(\xb-\xa,\yb-\ya)}
        \edef\ll{\pgfmathresult}
        \pgfmathparse{(\ll/2)-\rr}
        \edef\rc{\pgfmathresult}
        \begin{scope}
            \clip (\xc,\yc) circle (\rc);
            \draw[#1] (\xa,\ya) -- (\xb,\yb);
        \end{scope}
    \end{scope}
}

\def\tomrectangle[#1](#2,#3,#4,#5){
    \begin{scope}
        \edef\rr{\reccircsize}
        \pgfmathparse{#2-0.5}
        \edef\xa{\pgfmathresult}
        \pgfmathparse{#3-0.5}
        \edef\xb{\pgfmathresult}
        \pgfmathparse{#4-0.5}
        \edef\ya{\pgfmathresult}
        \pgfmathparse{#5-0.5}
        \edef\yb{\pgfmathresult}

        \ifthenelse{\equal{\xa}{\xb}}{
            \ifthenelse{\equal{\ya}{\yb}}{
                \draw[rectangle circle] (\xa,\ya) circle (\rr);
            }{
                \tomrectangleside[rectangle line](\xa,\ya,\xa,\yb)
                \draw[rectangle circle] (\xa,\ya) circle (\rr);
                \draw[rectangle circle] (\xa,\yb) circle (\rr);
            }
        }{
            \ifthenelse{\equal{\ya}{\yb}}{
                \tomrectangleside[rectangle line](\xa,\ya,\xb,\ya)
                \draw[rectangle circle] (\xa,\ya) circle (\rr);
                \draw[rectangle circle] (\xb,\ya) circle (\rr);
            }{
                \tomrectangleside[rectangle line](\xa,\ya,\xb,\ya)
                \tomrectangleside[rectangle line](\xa,\yb,\xb,\yb)
                \tomrectangleside[rectangle line](\xa,\ya,\xa,\yb)
                \tomrectangleside[rectangle line](\xb,\ya,\xb,\yb)
                \draw[rectangle circle] (\xa,\ya) circle (\rr);
                \draw[rectangle circle] (\xb,\ya) circle (\rr);
                \draw[rectangle circle] (\xa,\yb) circle (\rr);
                \draw[rectangle circle] (\xb,\yb) circle (\rr);
            }
        }
    \end{scope}
}

\def\tomlevelset[#1](#2,#3,#4,#5){
    \begin{scope}
        \edef\rr{\reccircsize}
        \pgfmathparse{#2-0.5}
        \edef\xa{\pgfmathresult}
        \pgfmathparse{#3-0.5}
        \edef\ya{\pgfmathresult}
        \pgfmathparse{#4-0.5}
        \edef\xb{\pgfmathresult}
        \pgfmathparse{#5-0.5}
        \edef\yb{\pgfmathresult}
        \pgfmathparse{(\xa+\xb)/2}
        \edef\xc{\pgfmathresult}
        \pgfmathparse{(\ya+\yb)/2}
        \edef\yc{\pgfmathresult}
        \pgfmathparse{veclen(\xb-\xa,\yb-\ya)}
        \edef\ll{\pgfmathresult}
        \pgfmathparse{(\ll/2)-\rr}
        \edef\rc{\pgfmathresult}
        \begin{scope}
            \clip (\xc,\yc) circle (\rc);
            \draw[#1] (\xa,\ya) -- (\xb,\yb);
        \end{scope}
    \end{scope}
}

\def\tomreidfig[#1](#2,#3,#4,#5,#6){
    \begin{scope}
        \edef\aa{#2} 
        \edef\bb{#3}
        \edef\cc{#4}
        \edef\dd{#5}
        \edef\ee{#6}
        \edef\un{0} 
        \tomrectangle[#1](\cc,\un,\dd,\bb);
        \tomrectangle[#1](\ee,\bb,\aa,\un);
        \tomlevelset[#1,levelset consequent](\cc,\dd,\ee,\aa);
        \tomlevelset[#1,levelset antecedent](\un,\dd,\bb,\aa);
        \tomlevelset[#1,levelset antecedent](\cc,\bb,\ee,\un);
        \tomlevelset[#1,levelset antecedent](\un,\bb,\bb,\un);
    \end{scope}
}

\edef\tomtableIX{
    {0,t,u,v,w,x,y,z,1},
    {0,0,t,u,u,v,w,x,z},
    {0,0,0,t,t,u,u,v,y},
    {0,0,0,t,t,u,u,v,x},
    {0,0,0,0,0,t,t,u,w},
    {0,0,0,0,0,t,t,u,v},
    {0,0,0,0,0,0,0,t,u},
    {0,0,0,0,0,0,0,0,t},
    {0,0,0,0,0,0,0,0,0}}

\edef\tomtableVIII{
    {0,u,v,w,x,y,z,1},
    {0,0,u,u,v,w,x,z},
    {0,0,0,0,u,u,v,y},
    {0,0,0,0,u,u,v,x},
    {0,0,0,0,0,0,u,w},
    {0,0,0,0,0,0,u,v},
    {0,0,0,0,0,0,0,u},
    {0,0,0,0,0,0,0,0}}

\edef\tomtableVII{
    {0,v,w,x,y,z,1},
    {0,0,0,v,w,x,z},
    {0,0,0,0,0,v,y},
    {0,0,0,0,0,v,x},
    {0,0,0,0,0,0,w},
    {0,0,0,0,0,0,v},
    {0,0,0,0,0,0,0}}

\edef\tomtableVI{
    {0,w,x,y,z,1},
    {0,0,0,w,x,z},
    {0,0,0,0,0,y},
    {0,0,0,0,0,x},
    {0,0,0,0,0,w},
    {0,0,0,0,0,0}}

\edef\tomtableV{
    {0,x,y,z,1},
    {0,0,0,x,z},
    {0,0,0,0,y},
    {0,0,0,0,x},
    {0,0,0,0,0}}

\edef\tomtableIV{
    {0,y,z,1},
    {0,0,0,z},
    {0,0,0,y},
    {0,0,0,0}}

\edef\tomtableIII{
    {0,z,1},
    {0,0,z},
    {0,0,0}}

\edef\tomtableII{
    {0,1},
    {0,0}}

\edef\tomtableI{
    {1}}

\edef\tomsupparch{
    {0,a,1,1,1,1,1,1,1,1,1},
    {0,0,0,1,1,1,1,1,1,1,1},
    {0,0,0,0,0,1,1,1,1,1,1},
    {0,0,0,0,0,0,1,1,1,1,1},
    {0,0,0,0,0,0,0,1,1,1,1},
    {0,0,0,0,0,0,0,0,1,1,1},
    {0,0,0,0,0,0,0,0,1,1,1},
    {0,0,0,0,0,0,0,0,0,1,1},
    {0,0,0,0,0,0,0,0,0,0,1},
    {0,0,0,0,0,0,0,0,0,0,a},
    {0,0,0,0,0,0,0,0,0,0,0}}

\edef\tomsuppnonarch{
    {0,a,1,1,1,1,1,1,1,1,1},
    {0,0,0,1,1,1,1,1,1,1,1},
    {0,0,0,0,0,1,1,1,1,1,1},
    {0,0,0,0,0,0,1,1,1,1,1},
    {0,0,0,0,0,0,0,1,1,1,1},
    {0,0,0,0,0,0,0,1,1,1,1},
    {0,0,0,0,0,0,0,0,1,1,1},
    {0,0,0,0,0,0,0,0,a,a,1},
    {0,0,0,0,0,0,0,0,a,a,1},
    {0,0,0,0,0,0,0,0,a,a,a},
    {0,0,0,0,0,0,0,0,0,0,0}}

\newcommand{\figpartition}{
    \begin{tikzpicture}[scale=0.45]
        \begin{scope}[shift={(0,0)}]
            \tomaxis{\tomtableIX}
            \tomtablefilled{\tomtableIX}
        \end{scope}
    \end{tikzpicture}
}

\newcommand{\figrees}{
    \begin{tikzpicture}[scale=0.4]
        \begin{scope}[shift={(0,0)}]
            \tomaxis{\tomtableVIII}
            \tomtablefilled{\tomtableVIII}
        \end{scope}
        \begin{scope}[shift={(10,0)}]
            \tomaxis{\tomtableVII}
            \tomtablefilled{\tomtableVII}
        \end{scope}
        \begin{scope}[shift={(19,0)}]
            \tomaxis{\tomtableVI}
            \tomtablefilled{\tomtableVI}
        \end{scope}
        \begin{scope}[shift={(-3,-11)}]
            \tomaxis{\tomtableV}
            \tomtablefilled{\tomtableV}
        \end{scope}
        \begin{scope}[shift={(4,-11)}]
            \tomaxis{\tomtableIV}
            \tomtablefilled{\tomtableIV}
        \end{scope}
        \begin{scope}[shift={(10,-11)}]
            \tomaxis{\tomtableIII}
            \tomtablefilled{\tomtableIII}
        \end{scope}
        \begin{scope}[shift={(15,-11)}]
            \tomaxis{\tomtableII}
            \tomtablefilled{\tomtableII}
        \end{scope}
        \begin{scope}[shift={(19,-11)}]
            \tomaxis{\tomtableI}
            \tomtablefilled{\tomtableI}
        \end{scope}
    \end{tikzpicture}
}

\newcommand{\figreidemeister}{
    \begin{tikzpicture}[scale=0.45]
        \begin{scope}
            \edef\tomsize{11}
            \edef\ax{1} 
            \edef\aax{2} 
            \edef\un{0} 
            \edef\aa{-6} 
            \edef\bb{-3}
            \edef\cc{-5}
            \edef\dd{-8}
            \edef\ee{-7}
            \tomtable{\tomsize}
            \tomunit{\tomsize}
            \tomzero{\tomsize}
            \tomcellfill[color=white, fill=white, ultra thick](-7.2,-9.8,-6.7,-7.3)
            \tomcellfill[color=white, fill=white, ultra thick](-5.2,-7.8,-8.7,-9.3)
            \tomreidfig[rectangle in table](\aa,\bb,\cc,\dd,\ee);
            \tomcellwriteright[](\ax,\dd,$a\mo b$)
            \tomcellwrite[](\ax,\aa,$a$)
            \tomcellwrite[](\ax,\bb,$b$)
            \tomcellwrite[](\bb,\ax,$b$)
            \tomcellwrite[](\cc,\ax,$c$)
            \tomcellwrite[](\un,\bb,$b$)
            \tomcellwrite[](\un,\dd,$d$)
            \tomcellwrite[](\bb,\un,$b$)
            \tomcellwrite[](\ee,\un,$e$)
            \tomcellwrite[](\bb,\aa,$d$)
            \tomcellwrite[](\cc,\bb,$e$)
            \tomcellwrite[](\ee,\ax,$b\mo c$)
            \edef\temptext{$(a\mo b)\mo c$}
            \tomcellwriteleft[](\cc,\dd-1,\temptext)
            \edef\temptext{$a\mo(b\mo c)$}
            \tomcellwriteleft[](\ee,\aa-1,\temptext)
        \end{scope}
    \end{tikzpicture}
}

\newcommand{\figarch}{
    \begin{tikzpicture}[scale=0.45]
        \begin{scope}
            \foreach \tomrow [count=\tomsize] in \tomsupparch {}
            \pgfmathparse{\tomsize+3}
            \edef\tomshift{\pgfmathresult}
            \begin{scope}[shift={(\tomshift,0)}]
                \edef\ax{1} 
                \edef\un{0} 
                \edef\ca{-1} 
                \pgfmathparse{1-\tomsize}
                \edef\ze{\pgfmathresult} 
                \pgfmathparse{2-\tomsize}
                \edef\at{\pgfmathresult} 
                \edef\aa{-5} 
                \edef\bb{-3}
                \edef\cc{\ca}
                \edef\dd{\at}
                \edef\ee{-5}
                \tomsupp{\tomsupparch}
                \tomcellwrite[](\ax,\bb,$b$)
                \tomcellwrite[](\ax,\aa,$a$)
                \tomcellwrite[](\bb,\ax,$b$)
                \tomcellwrite[](\ee,\ax,$e$)
                \tomcellwrite[](\cc,\ax,$c$)
                \tomcellwrite[](\un,\bb,$b$)
                \tomcellwrite[](\un,\dd,$\atom$)
                \tomcellwrite[](\bb,\un,$b$)
                \tomcellwrite[](\ee,\un,$e$)
                \tomcellwrite[](\bb,\aa,$\atom$)
                \tomcellwrite[](\cc,\bb,$e$)
                \tomcellwrite[](\ax,\dd,$d$)
                \tomcellwrite[lvl 0](\ee,\aa,$0$)
                \tomcellwrite[lvl 0](\cc,\dd,$0$)
                \tomreidfig[](\aa,\bb,\cc,\dd,\ee)
                \tomunit{\tomsize}
                \tomzero{\tomsize}
                \tomtable{\tomsize}
            \end{scope}
            \begin{scope}
                \edef\ax{1} 
                \edef\un{0} 
                \edef\aa{-4} 
                \edef\bb{-2}
                \edef\cc{-4}
                \edef\dd{-7}
                \edef\ee{-6}
                \tomsupp{\tomsupparch}
                \tomcellwrite[](\ax,\bb,$b$)
                \tomcellwrite[](\ax,\dd,$d$)
                \tomcellwrite[](\ax,\aa,$a$)
                \tomcellwrite[](\bb,\ax,$b$)
                \tomcellwrite[](\ee,\ax,$e$)
                \tomcellwrite[](\cc,\ax,$c$)
                \tomcellwrite[](\un,\bb,$b$)
                \tomcellwrite[](\un,\dd,$d$)
                \tomcellwrite[](\bb,\un,$b$)
                \tomcellwrite[](\ee,\un,$e$)
                \tomcellwrite[](\bb,\aa,$d$)
                \tomcellwrite[](\cc,\bb,$e$)
                \tomzero{\tomsize}
                \tomtable{\tomsize}
                \tomreidfig[](\aa,\bb,\cc,\dd,\ee);
                \tomunit{\tomsize}
            \end{scope}
        \end{scope}
    \end{tikzpicture}
}

\newcommand{\fignonarch}{
    \begin{tikzpicture}[scale=0.45]
        \begin{scope}
            \foreach \tomrow [count=\tomsize] in \tomsuppnonarch {}
            \pgfmathparse{\tomsize+3}
            \edef\tomshift{\pgfmathresult}
            \begin{scope}
                \edef\ax{1} 
                \edef\un{0} 
                \pgfmathparse{2-\tomsize}
                \edef\at{\pgfmathresult} 
                \edef\pp{-2} 
                \edef\aa{-5} 
                \edef\bb{-4}
                \pgfmathparse{\pp-1}
                \edef\cc{\pgfmathresult}
                \edef\dd{\at}
                \edef\ee{-7}
                \tomsupp{\tomsuppnonarch}
                \tomcellwrite[](\ax,\bb,$b$)
                \tomcellwrite[](\ax,\aa,$a$)
                \tomcellwrite[](\ax,\dd,$d$)
                \tomcellwrite[](\bb,\ax,$b$)
                \tomcellwrite[](\pp,\ax,$\idp_r$);
                \tomcellwrite[](\ee,\ax,$e$)
                \tomcellwrite[](\cc,\ax,$c$)
                \tomcellwrite[](\un,\bb,$b$)
                \tomcellwrite[](\un,\dd,$\atom$)
                \tomcellwrite[](\bb,\un,$b$)
                \tomcellwrite[](\ee,\un,$e$)
                \tomcellwrite[](\bb,\aa,$\atom$)
                \tomcellwrite[](\cc,\bb,$e$)
                \tomcellwrite[lvl 0](\ee,\aa,$0$)
                \tomcellwrite[lvl 0](\cc,\dd,$0$)
                \tomreidfig[rectangle in table](\aa,\bb,\cc,\dd,\ee);
                \tomunit{\tomsize}
                \tomzero{\tomsize}
                \tomtable{\tomsize}
            \end{scope}
            \begin{scope}[shift={(\tomshift,0)}]
                \edef\ax{1} 
                \edef\un{0} 
                \pgfmathparse{1-\tomsize}
                \edef\ze{\pgfmathresult} 
                \pgfmathparse{2-\tomsize}
                \edef\at{\pgfmathresult} 
                \edef\pp{-2} 
                \edef\aa{-7} 
                \edef\bb{-4}
                \edef\cc{\pp}
                \edef\dd{\at}
                \edef\ee{-6}
                \tomsupp{\tomsuppnonarch}
                \tomcellwrite[](\ax,\bb,$b$)
                \tomcellwrite[](\ax,\dd,$\atom$)
                \tomcellwrite[](\pp,\ax,$c{\footnotesize =}\idp_r$);
                \tomcellwrite[](\ax,\aa,$a$)
                \tomcellwrite[](\bb,\ax,$b$)
                \tomcellwrite[](\ee,\ax,$e$)
                \tomcellwrite[](\un,\bb,$b$)
                \tomcellwrite[](\bb,\un,$b$)
                \tomcellwrite[](\ee,\un,$e$)
                \tomcellwrite[](\cc,\bb,$e$)
                \tomzero{\tomsize}
                \tomreidfig[rectangle in table, color=gray](\aa,\bb,\cc,\ze,\ee);
                \tomcellwrite[](\un,\dd,$\atom$)
                \tomcellwrite[](\cc,\dd,$\atom$)
                \tomcellwrite[color=black](\un,\ze,$0$)
                \tomcellwrite[color=black](\cc,\ze,$0$)
                \tomtable{\tomsize}
                \tomreidfig[rectangle in table](\aa,\bb,\cc,\dd,\ee);
                \tomunit{\tomsize}
            \end{scope}
        \end{scope}
    \end{tikzpicture}
}

\edef\tomtableCntexParent{
    {0,y,z,1},
    {0,0,z,z},
    {0,0,y,y},
    {0,0,0,0}}

\edef\tomtableCntexStepA{
    {0,a,y,z,1},
    {0,a,a,z,z},
    {0,0,?,y,y},
    {0,0,0,0,a},
    {0,0,0,0,0}}

\edef\tomtableCntexStepB{
    {0,a,y,z,1},
    {0,a,0,z,z},
    {0,0,?,y,y},
    {0,0,0,0,a},
    {0,0,0,0,0}}

\newcommand{\figcntex}{
    \begin{tikzpicture}[scale=0.45]
        \edef\zax{-3}
        \edef\abc{2.5}
        \edef\de{2}
        \begin{scope}[shift={(0,0)}]
            \tomaxis{\tomtableCntexParent}
            \tomtablefilled{\tomtableCntexParent}
        \end{scope}
        \edef\zax{-4}
        \begin{scope}[shift={(7,0)}]
            \tomaxis{\tomtableCntexStepA}
            \tomtablefilled{\tomtableCntexStepA}
        \end{scope}
        \begin{scope}[shift={(14,0)}]
            \edef\un{0} 
            \edef\aa{-1} 
            \edef\bb{-2}
            \edef\cc{-1}
            \edef\dd{-2}
            \edef\ee{-3}
            \tomaxis{\tomtableCntexStepB}
            \tomtablefilled{\tomtableCntexStepB}
            \tomrectangle[rectangle in table](\bb,\dd,\un,\cc);
            \tomrectangle[rectangle in table](\un,\aa,\bb,\ee);
        \end{scope}
    \end{tikzpicture}
}

\begin{document}

\title{Rees coextensions \\
of finite, negative tomonoids
\thanks{This is a pre-copyedited, author-produced version of an article accepted for publication in the {\it Journal of Logic and Computation} following peer review. The version of record,
M.~Petr\' ik, Th.~Vetterlein, Rees coextensions of finite, negative tomonoids,
{\sl J.~Log.~Comput.}~{\bf 27} (2017), 337-356,
is available online at {\tt https://doi.org/10.1093/logcom/exv047}.}}

\author{Milan Petr\'{\i}k}

\affil{\footnotesize
    Department of Mathematics, 
    Faculty of Engineering,\\ 
    Czech University of Life Sciences, 
    Prague, Czech Republic;\\
    Institute of Computer Science, \\
    Academy of Sciences of the Czech Republic, Prague, Czech Republic\\
{\tt petrik@cs.cas.cz}}

\author{Thomas Vetterlein}

\affil{\footnotesize
    Department of Knowledge-Based Mathematical Systems,\\
    Johannes Kepler University, Linz, Austria\\
{\tt Thomas.Vetterlein@jku.at}}

\maketitle

\thispagestyle{plain}


\begin{abstract}

\vspace{-3ex}\mbox{}\parindent0pt\parskip1ex

A totally ordered monoid, or tomonoid for short, is a monoid endowed with a compatible total order. We deal in this paper with tomonoids that are finite and negative, where negativity means that the monoidal identity is the top element. Examples can be found, for instance, in the context of finite-valued fuzzy logic.

By a Rees coextension of a negative tomonoid $S$, we mean a negative tomonoid $T$ such that a Rees quotient of $T$ is isomorphic to $S$. We characterise the set of all those Rees coextensions of a finite, negative tomonoid that are by one element larger. We thereby define a method of generating all such tomonoids in a stepwise fashion. Our description relies on the level-set representation of tomonoids, which allows us to identify the structures in question with partitions of a certain type.


\end{abstract}

\section{Introduction}

A totally ordered monoid, or tomonoid as we say shortly \cite{EKMMW}, is a monoid $(\So \sep \mo, 1)$ endowed with a total order $\leq$ that is compatible with the monoidal operation. The compatibility means that, for any $a, b, c \in \So$, $a \leq b$ implies $a \mo c \leq b \mo c$ and $c \mo a \leq c \mo b$. Tomonoids occur in a number of different contexts. For instance, the term orders used in computational mathematics in connection with Gr\"{o}bner bases can be identified with positive orders on $\Naturals^n$ that are compatible with the addition. Hence term orders correspond to total orders making $\Naturals^n$ into a tomonoid \cite{CLS}. Numerous further examples can be found in the context of many-valued logics. In particular, fuzzy logics are based on an extended set of truth values and this set is usually totally ordered \cite{Haj}. The conjunction in fuzzy logic is moreover commonly interpreted by an associative operation with which the total order is compatible. Hence we are naturally led to tomonoids.

The probably most familiar type of a tomonoid in fuzzy logic is the real unit interval endowed with the natural order, a left-continuous triangular norm, and the constant $1$ \cite{Haj,KMP}. Note that this tomonoid is negative and commutative. Negativity means that the monoidal identity is the top element. This condition will be assumed throughout the present paper as well. In contrast, our results are developed without the assumption of commutativity. We will see, however, that the adaptation to the commutative case does not cause difficulties.

A remarkable effort has been spent in the last decade on the problem of describing tomonoids in a systematic way, often under the assumption of negativity and mostly under the assumption of commutativity; we may, for example, refer to \cite{EKMMW,Hor1,Hor2}. In particular, the aforementioned triangular norms have been an intensive research field; see, e.g., \cite{NEG,Vet1}. Although the examination of negative, commutative tomonoids has made from an algebraic perspective a considerable progress, a comprehensive classification of these structures has not yet been found.

Given the complexity of the problem, it seems to be reasonable to consider separately the finite case. This is what we do in this paper. Immediate simplifications cannot be expected from this further restriction. Nothing seems to indicate that tomonoids are easier to describe under the finiteness assumption. Quite a few papers are devoted to finite tomonoids; see, for instance, \cite{Hor3,Vet2}.

The starting point of the present paper is the following simple observation. Let $(\So \sep \leq,$ $\mo, 1)$ be a finite negative tomonoid. Let $0$ be the bottom element of $\So$ and let $\atome$ be the atom of $\So$, that is, the smallest element apart from $\zeroe$. Then the identification of $\atome$ with the bottom element of $\So$ is a tomonoid congruence. With regard to the semigroup reduct, this is the Rees congruence by the ideal $\{ \zeroe, \atome \}$. The quotient is by one element smaller than $S$ and forming the same kind of quotient repeatedly, we get a sequence of tomonoids eventually leading to the tomonoid that consists of the single element $1$.

Seen from the other direction, each $n$-element negative tomonoid is the last entry in a sequence of $n$ tomonoids the first of which is the one-element tomonoid and each other leads to its predecessor by the identification of its smallest two elements. Proceeding one step forward in this sequence means replacing a finite, negative tomonoid by a tomonoid that is by one element larger and whose Rees quotient by the ideal consisting of the bottom element and the atom is the original one. In this paper, we specify all possibilities of enlarging the tomonoid in this way. Therefore, we propose a way of generating systematically all finite, negative tomonoids.

In accordance with the unordered case \cite{Gri2}, we call a tomonoid whose quotient is a tomonoid $S$ a {\it coextension} of $S$. We deal with Rees congruences, which are understood as usual but restricted to the case that the ideal is a downward closed set. Accordingly, we have chosen the notion of a {\it Rees coextension} to name our construction. Explicitly, a Rees coextension of a negative tomonoid $S$ is a negative tomonoid $T$ such that $S$ is isomorphic to a Rees quotient of $T$. We focus on the case that the cardinality of the coextension is just by $1$ larger and speak about {\it one-element} (Rees) coextensions then.

Coextensions of semigroups have been explored under various conditions, which are usually, however, in the present context quite special. For instance, coextensions of regular semigroups have been considered in \cite{MeNa}. In contrast, our present work is closely related to the theory of extensions of semigroups. A semigroup $T$ is an {\it (ideal) extension} of a semigroup $I$ by a semigroup $S$ if $I$ is an ideal of $T$ and $S$ is the Rees quotient of $T$ by $I$. Semigroup extensions were first investigated in \cite{Cli}; see also \cite[Section 4.4]{ClPr} or \cite[Chapter 3]{Pet}. It is moreover straightforward to adapt the notion to the ordered case \cite{Hul,KeTs}. What we study in this paper are actually special ideal extensions. The latter terminology just reflects a different viewpoint: the extended and extending semigroups are denoted in the opposite way. That is, we study ideal extensions of a two-element semigroup by a finite negative tomonoid.

Ideal extensions of ordered semigroups were first studied by A.\ J.\ Hulin. In \cite{Hul}, Clifford's technique of constructing extensions by means of partial homomorphisms was adapted to the ordered case. However, if the extended semigroup does not possess an identity the method does not necessarily cover all possible extensions. A more general method, which is applicable to weakly reductive semigroups, is due to Clifford as well. The presumed condition, however, although found ``relatively mild'' in \cite{ClPr}, turns out to be quite special in the present context again. The ordered case was investigated along these lines by N.\ Kehayopulu and M.\ Tsingelis in \cite{KeTs}.

Let us have a closer look at the method that we are going to discuss here. The construction requires the duplication of the bottom element; the latter is replaced a new bottom element and a new atom. Then, the multiplication needs to be revised in all those cases that lead, in the original tomonoid, to the bottom element. Trying out some simple examples, we soon observe that this problem is more difficult than it looks. Because of the mutual interdependencies, to decide which pairs of elements multiply to the (new) bottom element and which pairs multiply to the atom is not straightforward.

To bring transparency into this problem, a framework in which the structures under consideration become manageable is desirable. The crucial property with which we have to cope is associativity. This property is fundamental in mathematics and numerous approaches exist to shed light on it. Let us enumerate some ideas that are applicable in our context.

\begin{itemize}

\item We can lead back the associativity of a monoid to the probably most common situation where this property arises: the addition of natural numbers. Naturally, this approach is limited to the commutative case. In fact, any commutative monoid, provided it is finitely generated, is a quotient of $\mathbb{N}^n$, where $n$ is the number of generators. A description of tomonoids on this basis has been proposed, e.g., in\ \cite{Vet2}.

\item There is another situation in which associativity arises naturally: the composition of functions. In fact, we may represent any monoid as a monoid of mappings under composition. Namely, we may use the regular representation; see, e.g., \cite{ClPr}. In the presence of commutativity, we are led to a monoid of pairwise commuting, order-preserving mappings. The associativity is then accounted for by the fact that any two mappings commute. This idea is applied to tomonoids in \cite{Vet1}.

\item A third and once again totally different approach is inspired by the field of web geometry; see, e.g., \cite{Acz,BlBo}. Here, a tomonoid is represented by its level sets. Associativity then corresponds to the so-called Reidemeister condition. This approach has been applied to triangular norms in \cite{PeSa1,PeSa2}.

\end{itemize}

For our aims, any of these three approaches is worth being considered. Each of them has its benefits and drawbacks. The present paper is devoted to the third approach.

We may represent any two-place function by means of its level sets. The idea is simple and means in our context the following. Let $(\So \sep \leq)$ be a chain, that is, a totally ordered set. Let $\mo \colon \So \times \So \to \So$ be a binary operation on $\So$, and consider the following equivalence relation on the set $\So^2 = \So \times \So$:
\begin{eqnarray*}
(a,b) \sim (c,d) & \textrm{if} & a \mo b = c \mo d.
\end{eqnarray*}
Then $\sim$ partitions $\So^2$ into the subsets of pairs that are assigned equal values. To recover $\mo$ from $\sim$, all we need to know is which subset is associated with which value of $\So$. But if we know that $\mo$ behaves neutrally with respect to a designated element $1$ of $\So$, this is clear: each class then contains exactly one element of the form $(1,a)$ and is associated with $a$. Consequently, a tomonoid $(\So \sep \leq, \mo, 1)$ can be identified with a chain $\So$ together with a certain partition on $\So^2$ and the designated element $1$.

To determine the one-element Rees coextensions means, in this picture, to replace the partition on $\So^2$ by a suitable partition on $\Se^2$, where $\Se$ is the chain arising from $\So$ by a duplication of the bottom element. Thus our topic is to specify a procedure leading exactly to those partitions of the enlarged set $\Se^2$ that correspond to the one-element Rees coextensions.

We proceed as follows. In Section \ref{sec:Tomonoids}, we specify the structures under consideration and the particular type of quotients that we employ. In Section \ref{sec:Level-sets}, we put up our basic framework, which relies on the level-set representation of binary operations. In Section \ref{sec:Archimedean-extensions}, we start describing how the Rees coextensions of negative tomonoids can be determined. In this first step, we restrict to the case that the tomonoid is Archimedean. The general case is discussed in the subsequent Section \ref{sec:Extensions}. Section \ref{sec:Conclusion} contains some concluding remarks.

\section{Totally ordered monoids} \label{sec:Tomonoids}

We investigate in this paper the following structures.

\begin{definition} \label{def:Tomonoids}

A \emph{totally ordered monoid}, or a \emph{tomonoid} for short, is a structure $(\So \sep \leq, \linebreak \mo, 1)$ such that $(\So \sep \mo, 1)$ is a monoid, $(\So \sep \leq)$ is a chain, and $\leq$ is compatible with $\mo$, that is, for any $a, b, c \in \So$, $\,a \leq b$ implies $a \mo c \leq b \mo c$ and $c \mo a \leq c \mo b$. We call a tomonoid $(\So \sep \leq, \mo, 1)$ \emph{negative} if $1$ is the top element, and we call $\So$ \emph{commutative} if so is $\mo$.
\end{definition}

We are exclusively interested in tomonoids that are finite and negative. We abbreviate these properties by ``f.n.''. We note that, in the context of residuated lattices, the notion ``integral'' is commonly used instead of ``negative''. We further note that, in contrast to \cite{EKMMW}, we do not assume a tomonoid to be commutative and in fact we proceed without this assumption. However, the commutative case is without doubt important and will be considered as well.

The smallest tomonoid is the one that consists of the monoidal identity $1$ alone, called the \emph{trivial} tomonoid. Tomonoids with at least two elements are called \emph{non-trivial}.

Congruences of tomonoids are defined as follows; cf.\ \cite{EKMMW}. Here, a subset $C$ of a poset is called {\it convex} if $a, c \in C$ and $a \leq b \leq c$ imply $b \in C$.

\begin{definition} \label{def:Tomonoid-congruence}

Let $(\So \sep \leq, \mo, 1)$ be a tomonoid. A \emph{tomonoid congruence} on $\So$ is an equivalence relation $\congruence$ on $\So$ such that 
(i) $\congruence$ is a congruence of $\So$ as a monoid and
(ii) each $\congruence$-class is convex.
On the quotient $\congrcl{\So}$, we then denote the operation induced by $\mo$ again by $\mo$ and, for $a, b \in \So$, we let $\congrcl{a} \leq \congrcl{b}$ if $a \congruence b$ or $a < b$.

\end{definition}

If $\congruence$ is a tomonoid congruence on a tomonoid $\So$, we easily check that $(\congrcl{\So} \sep \leq,  \mo, \linebreak \congrcl{1})$ is a tomonoid again, called the \emph{quotient} of $\So$ by $\congruence$. It is clear that the formation of a quotient preserves the properties of finiteness, negativity, and commutativity, respectively.

In \cite{Vet1}, congruences of negative, commutative tomonoids are discussed that are induced by filters. Provided that a tomonoid is residuated, these congruences are precisely those that also preserve the residual implication. Here, we consider something different. A particularly simple type of congruences is the following; see, e.g., \cite{EKMMW}.

\begin{lemma} \label{lem:Rees-quotient}

Let $(\So \sep \leq, \mo, 1)$ be a negative tomonoid and let $q \in \So$. For $a, b \in \So$, let $a \rees_q b$ if $a  =  b$ or $a, b \leq q$. Then $\rees_q$ is a tomonoid congruence.

\end{lemma}

Note that, because of the negativity of $\So$, the set $\{ a \in \So \colon a \leq q \}$ in Lemma \ref{lem:Rees-quotient} is a semigroup ideal of $\So$. This is why the indicated congruence is actually a Rees congruence of $\So$, seen as a semigroup; see, e.g., \cite{How}.

For a finite chain $\So$, let $0$ denote the bottom element. We write $\withoutzero{\So} = \So \setminus \{0\}$. Furthermore, we call the second smallest element of $\So$, if it exists, the \emph{atom} of $\So$. The symbol $\atom$ will be used in the sequel to denote it.

\begin{definition} \label{def:Rees-quotient}

Let $(\So \sep \leq, \mo, 1)$ be a {\fntomonoid} and let $q \in \So$. Then we call $\rees_q$, as defined in Lemma \ref{lem:Rees-quotient}, the {\it Rees congruence} by $q$. We denote the quotient by $\So/q$ and call it the {\it Rees quotient} of $\So$ by $q$.

Moreover, we call $\So$ a {\it Rees coextension} of $\So/q$. We call $S$ a {\it one-element} Rees coextension, or simply a {\it one-element coextension}, if $\So$ is non-trivial and $q$ is the atom of $\So$.

\end{definition}

The problem that we rise in this paper is: How can we determine all one-element coextensions of a {\fntomonoid}? Having defined a suitable such method, we will obviously be in the position to determine, starting from the trivial tomonoid, successively all {\fntomonoid}s.

\section{Tomonoid partitions} \label{sec:Level-sets}

Let $\binaryoperation$ be a binary operation on a set $A$. Then $\binaryoperation$ gives rise to a partition of $A \times A$: the blocks of the partition are the subsets of all those pairs that are mapped by $\binaryoperation$ to the same value. This partition, together with the assignment that associates with each block the respective element of $A$, specifies $\binaryoperation$ uniquely.

The representation of tomonoids that we will employ in the sequel is based on this simple idea. From a geometric point of view, we will deal with a representation of tomonoids that comes along with two dimensions---in contrast to the commonly used graph of binary operations. For the case of triangular norms, the idea has been proposed in \cite{PeSa1} and in this framework an open problem on the convex combinations of t-norms was solved \cite{PeSa2}. An adaptation to the present context causes no difficulties.

We deal in the sequel with partitionings of posets. Let us fix some terminology. Let $(M \sep \leq)$ be a poset and let $\equ$ be an equivalence relation on $M$. Then $\leq$ induces the preorder $\preordersim$ on the set $\equcl{M}$ of $\equ$-classes, where, for $a, b \in M$,
\begin{align*}
\equcl{a} \preordersim \equcl{b} \quad & \text{if there are $c_0, \ldots, c_k \in M$ such that} \\
& a \equ c_0 \leq c_1 \equ c_2 \leq \ldots \leq c_k \equ b.
\end{align*}
We say that $\equ$ is {\it regular} for $\leq$ if the following condition, sometimes called the closed chain condition, is fulfilled: For any $c_0, \ldots, c_k \in M$ such that $c_0 \equ c_1 \leq c_2 \equ c_3 \leq \ldots \leq c_k \equ c_0$, we have $c_0 \equ \ldots \equ c_k$. In this case, $\preordersim$ is obviously antisymmetric and hence a partial order.

In other words, if an equivalence relation $\equ$ on a poset $(M \sep \leq)$ is regular for $\leq$, then $(\equcl{M} \sep \preordersim)$ is a poset again and the natural surjection $a \mapsto \equcl{a}$ is order-preserving. Our terminology originates from \cite[Def.\ 1.7]{Cod}. The paper \cite{Cod} in fact contains a detailed discussion of partitions of posets and their relationship to the partial order.

We now turn to our actual objects of interest.

\begin{definition}

Let $(\So \sep \leq, \mo, 1)$ be a tomonoid. For two pairs $(a,b), (c,d) \in \So^2$ we define
\[ (a,b) \eqo (c,d) \quad\textrm{if}\quad a \mo b = c \mo d, \]
and we call $\eqo$ the {\it level equivalence} of $\So$.

\end{definition}

The level equivalence of a tomonoid $\So$ defines a certain partition of $\So^2$. We define a corresponding relational structure.

\begin{definition} \label{def:tomonoid-partition}

Let $(\So \sep \leq)$ be a chain and let $1 \in \So$. By $\leqcomp$, we denote the componentwise order on $S^2$, that is, we put
\[ (a,b) \leqcomp (c,d) \quad\text{if $a \leq c$ and $b \leq d$} \]
for $a, b, c, d \in \So$. Moreover, let $\eqo$ be an equivalence relation on $\So^2$ such that the following conditions hold:
\begin{itemize}
\axiom{P1} $\eqo$ is regular for $\leqcomp$.
\axiom{P2} For any $(a, b) \in \So^2$ there is exactly one $c \in \So$ such that
$(a,b) \eqo (1,c) \eqo (c,1)$.
\axiom{P3} For any $a, b, c, d, e \in \So$, $\; (a,b) \eqo (d,1)$ and $(b,c) \eqo (1,e)$ imply $(d,c) \eqo (a,e)$.
\end{itemize}
Then the structure $(\So^2 \sep \leqcomp, \eqo, \oneone)$ is called a \emph{tomonoid partition}.

\end{definition}

\begin{proposition} \label{prop:from-tomonoid-to-partition}
Let $(\So \sep \leq, \mo, 1)$ be a tomonoid and let $\eqo$ be the level equivalence of $\So$. Then $(\So^2 \sep \leqcomp, \eqo, \oneone)$ is a tomonoid partition.

\end{proposition}

\begin{proof}
Let $a, b, c, d \in S$. By the compatibility of $\leq$ with $\mo$, we have that $(a,b) \leqcomp (c,d)$ implies $a \mo b \leq c \mo d$. (P1) follows. Moreover, as $1$ is the monoidal identity, we have that $(a,b) \eqo (c,1)$ iff $(a,b) \eqo (1,c)$ iff $a \mo b = c$. Hence also (P2) holds. Finally, (P3) is implied by the associativity of $\mo$.
\end{proof}

In the sequel, given a tomonoid $(\So \sep \leq, \mo, 1)$ and its level equivalence $\eqo$, we refer to $(\So^2 \sep \leqcomp, \eqo, \oneone)$ as the tomonoid partition {\it associated with $\So$}.

We note that, thanks to (P2), the regularity condition (P1) can be simplified.

\begin{lemma} \label{lem:regularity}

Let $(\So \sep \leq)$ be a chain, let $1 \in \So$, and let $\eqo$ be an equivalence relation on $\So^2$ such that {\rm (P2)} holds. Then {\rm (P1)} is equivalent to each of following statements:
\begin{itemize}
\axiom{P1'} For any $a, a', b, b', c, c', d, d' \in \So$, $\; (a,b) \eqo (a',b') \leqcomp (c,d) \eqo (c',d') \leqcomp (a,b)$ implies $(a,b) \eqo (c,d)$.

\axiom{P1''} For any $a, b, c, d, e, f \in \So$, if $(1,e) \eqo (a,b) \leqcomp (c,d) \eqo (1,f)$, then $e \leq f$.
\end{itemize}

\end{lemma}

\begin{proof}
(P1) trivially implies (P1').

Assume (P1') and let $(1,e) \eqo (a,b) \leqcomp (c,d) \eqo (1,f)$. Then $f < e$ implies $(1,f) \leqcomp (1,e)$ and hence, by (P1') and (P2), $e = f$, a contradiction. (P1'') follows.

Assume (P1'') and let $a, b, c_1, \ldots, c_k, d_1, \ldots, d_k$ be such that
\[ (c_1,d_1) \eqo (c_2,d_2) \leqcomp (c_3,d_3) \eqo (c_4,d_4) \leqcomp \ldots \leqcomp (c_k,d_k) \eqo (c_1,d_1). \]
By (P2), there are $e_1, \ldots, e_k$ such that $(c_i,d_i) \eqo (1,e_i)$ for each $i$, and we conclude by (P1'') that $e_1 = e_2 \leq e_3 = \ldots \leq e_k = e_1$. It follows that the $(c_i,d_i)$ are pairwise $\eqo$-equivalent, and (P1) is shown.
\end{proof}

Before establishing the converse direction of Proposition \ref{prop:from-tomonoid-to-partition}---in fact, tomonoid partitions can be identified with tomonoids---let us interpret the properties (P1)--(P3) in Definition \ref{def:tomonoid-partition} from a geometric point of view.

For a tomonoid $\So$, let us view $\So^2$ as a square array; cf.\ Figure \ref{fig:tomonoid-partition}. Let the columns and rows be indexed by the elements of $\So$ such that the order goes to the right and upwards, respectively. For two elements $(a,b), (c,d) \in \So^2$, we then have $(a,b) \leqcomp (c,d)$ if $(c,d)$ is on the right above $(a,b)$. Moreover, $\So$ contains a designated element $1$. In case of the tomonoids on which we focus here, $1$ is the top element. In this case the element $(1,1)$ of $\So^2$ is located in the upper right corner.

\begin{figure}[ht!]

\begin{center}

\figpartition

\end{center}

\caption{A tomonoid partition associated with a $9$-element negative tomonoid $S$. Rows and columns of the array correspond to the elements of $\So$. Each square in the array thus corresponds to a pair $(a,b) \in \So^2$, where $a$ is the row index and $b$ is the column index. Moreover, the element of $\So$ indicated in the square $(a,b)$ is the product of $a$ and $b$ in $\So$. For instance, $z \mo v = u$. Finally, let $\eqo$ be the level equivalence. Then two squares belong to the same $\eqo$-class iff they contain the same symbol. For instance, the $\eqo$-class of $t$ comprises ten elements and the $\eqo$-class of $1$ is just a singleton.
\label{fig:tomonoid-partition}}

\end{figure}

The level equivalence of $\So$ induces a partition of $\So^2$. Condition (P2) implies that the blocks of this partition are in a one-to-one correspondence with the elements of the line indexed by $1$. In fact, when passing through the line $(1,c)$, $c \in \So$, from left to right, we meet each block exactly once. The same holds for the column indexed by $1$. By (P2), $(c,1)$ and $(1,c)$ are for each $c \in \So$ in the same block.

The partial order $\leqcomp$ induces a preorder $\preorderblock$ on the blocks. Condition (P1) ensures that $\preorderblock$ is a partial order as well. In fact, $\preorderblock$ is the total order inherited from $\So$ under the correspondence between the blocks and the line $(1,c)$, $c \in \So$. A way to see the meaning of (P1) is thus the following: when switching from any element of a block containing $(1,c)$ to the right or upwards, then we arrive at a block containing $(1,d)$ such that $d \geq c$.

\begin{figure}[ht!]
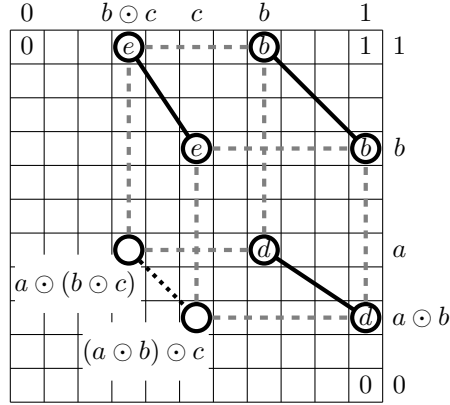


\begin{center}

\figreidemeister

\end{center}

\caption{The ``Reidemeister'' condition (P3). A (connected or broken) black line between two elements of the array indicates level equivalence; for instance, $(a,b) \eqo (a \mo b, 1)$.
By (P3), the equivalences of the pairs connected by a solid line imply the equivalence of the pair connected by a broken line.
\label{fig:P3}}

\end{figure}

Finally, assuming that the $1$ is the top element of $S$, also condition (P3) possesses an appealing interpretation in our geometric setting, a fact to which we will refer in the sequel repeatedly. Within the array representing $\So^2$, consider two rectangles such that one hits the upper edge and the other one hits the right edge; cf.\ Figure \ref{fig:P3}. Assume that the upper left, upper right, and lower right vertices of these rectangles are in the same blocks, respectively. Then, by (P3), also the remaining pair, consisting of the lower left vertices, is in the same block. The corresponding property in web geometry is the {\it Reidemeister condition} \cite{Acz,BlBo}.

In the sequel, when working with a set $\So^2$, where $\So$ is a chain with a designated element $1$, we will identify the elements of the form $(1,c)$, $c \in \So$, with $c$. It will be clear from the context if $c$ denotes an element of $\So$ or of $\So^2$. In particular, if $\eqo$ is an equivalence relation on $\So^2$, then $(a,b) \eqo c$ means $(a,b) \eqo (1,c)$. Moreover, the $\eqo$-class of a $c \in \So$ is meant to be the $\eqo$-class containing $(1,c)$.

\begin{proposition} \label{prop:from-partition-to-tomonoid}

Let $(\So^2 \sep \leqcomp, \eqo, \oneone)$ be a tomonoid partition. Let $\leq$ be the underlying total order of $\So$. Moreover, for any $a, b \in \So$, let
\begin{equation} \label{eq:def-mul-from-lvleq}
a  \mo b \; = \; \text{ the unique $c$ such that $(a,b) \eqo c$.}
\end{equation}
Then $(\So \sep \leq, \mo, 1)$ is the unique tomonoid such that $(\So^2 \sep \leqcomp, \eqo, \oneone)$ is its associated tomonoid partition.

\end{proposition}

\begin{proof}

By assumption, $\So$ is totally ordered and $\leqcomp$ is the induced componentwise order on $\So^2$. Evidently, $\leqcomp$ determines the total order $\leq$ on $\So$ uniquely. It is furthermore clear from (P2) that $\mo$ can be defined by (\ref{eq:def-mul-from-lvleq}).

For $a \in \So$, we have $1 \mo a = a$ by construction and $a \mo 1 = 1 \mo a$ by (P2). Furthermore, (P2) and (P3) imply the associativity of $\mo$. Thus $(\So \sep \mo, 1)$ is a monoid. Recall next that (P1) and (P2) imply (P1'') by Lemma \ref{lem:regularity}. Let $a \leq b$. Then $(a,c) \leqcomp (b,c)$, and we conclude from (P1'') that $a \mo c \leq b \mo c$. Similarly, we see that $c \mo a \leq c \mo b$. Thus $\leq$ is compatible with $\mo$ and $(\So \sep \leq, \mo, 1)$ is a tomonoid. It is clear that $\eqo$ is the level equivalence of $\So$ and we conclude that $(\So^2 \sep \leqcomp, \eqo, \oneone)$ is its associated tomonoid partition.

Let $(\So^2 \sep \leqcomp, \eqo, \oneone)$ be associated to another tomonoid $(\So' \sep \leq', \mo', 1')$. Then, by the way in which a tomonoid partition is constructed from a tomonoid, $\So' = \So$, $\;{\leq'} = {\leq}$, and $1' = 1$. Furthermore, if for some $a, b, c \in \So$ we have $a \mo' b = c$, then $(a,b) \eqo (1,c)$ and hence $a \mo b = c$. We conclude $\mo' = \mo$.
\end{proof}

By Propositions \ref{prop:from-tomonoid-to-partition} and \ref{prop:from-partition-to-tomonoid}, tomonoids and tomonoid partitions are in a one-to-one correspondence. We will present our results in the sequel mostly with reference to the latter, that is, with reference to tomonoid partitions.

Under this identification, we will apply properties, constructions, etc.\ defined for tomonoids to tomonoid partitions as well. For instance, a negative tomonoid partition is meant to be a tomonoid partition such that the corresponding tomonoid is negative.

Properties of tomonoids that are repeatedly addressed in this paper are finiteness, negativity, and commutativity. Finiteness has for tomonoids and their associated tomonoid partitions obviously the same meaning. Negativity and commutativity may be characterised as follows.

\begin{lemma} \label{lem:properties-of-tomonoid-partition}

Let $(\So^2 \sep \leqcomp, \eqo, \oneone)$ be a tomonoid partition.
\begin{itemize}

\item[\rm (i)] $\So^2$ is negative if and only if $(1,1)$ is the top element of $\So^2$ if and only if the $\eqo$-class of any $c \in \So$ is contained in $\{ (a,b) \in \So^2 \colon a, b \geq c \}$.

\item[\rm (ii)] $\So^2$ is commutative if and only if $(a,b) \eqo (b,a)$ for any $a, b \in \So$.

\end{itemize}

\end{lemma}

The following proposition is devoted to the structures in which we are actually interested: the {\fntomonoid} partitions. The slightly optimised characterisation will be useful in subsequent proofs.

\begin{proposition} \label{prop:PartAx-ass-simplified}
Let $(\So \sep \leq)$ be a finite and at least two-element chain with the top element $1$. Let $\zeroo$ be the bottom element of $\So$. Then $(\So^2 \sep \leqcomp, \eqo, \oneone)$ is a tomonoid partition if and only if {\rm (P1'')}, {\rm (P2)}, and the following condition hold:
\begin{itemize}
\axiom{P3'} For any $a, b, c, d, e \in \So \setminus \{ \zeroo, 1 \}$, $\; (a,b) \eqo d$ and $(b,c) \eqo e$ imply $(d,c) \eqo (a,e)$.
\end{itemize}
In this case, $(\So^2 \sep \leqcomp, \eqo, \oneone)$ is finite and negative.
\end{proposition}

\begin{proof}
The ``only if'' part is clear by definition and by Lemma \ref{lem:regularity}.

To see the ``if'' part, let $(\So^2 \sep \leqcomp, \eqo, \oneone)$ fulfil (P1''), (P2), and (P3'). Then (P1) holds by Lemma \ref{lem:regularity}. We next show that the negativity criterion of Lemma \ref{lem:properties-of-tomonoid-partition}(i) holds:

($\star$) $(a,b) \eqo (1,c)$ implies $c \leq a$ and $c \leq b$.

Indeed, in this case $(c,1) \eqo (1,c) \eqo (a,b) \leqcomp (a,1)$ by (P2) and the fact that $1$ is the top element. Hence, by (P1''), $c \leq a$. Similarly, we see that $c \leq b$.

It remains to prove (P3). Let $a, b, c, d, e \in \So$ be such that $(a,b) \eqo d$ and $(b,c) \eqo e$. We have to show $(d,c) \eqo (a,e)$ if one of the five elements equals $0$ or $1$. We consider certain cases only, the remaining ones are seen similarly.

Let $a = 1$. Then $(1,b) \eqo (1,d)$, hence $b = d$ by (P2), and it follows $(d,c) = (b,c) \eqo (1,e) = (a,e)$.

Let $d = 1$. Then $(a,b) \eqo (1,1)$, and by ($\star$), we conclude $a = b = 1$. From $(b,c) \eqo e$ it follows $e = c$. Hence $(d,c) = (a,e)$.

Note next that, for any $f \in \So$, $(f,0) \eqo 0$. This follows again from ($\star$).

Let $a = 0$. Then $(a,b) = (0,b) \eqo 0$ and hence $d = 0$. Hence $(d,c) = (0,c) \eqo 0 \eqo (0,e) = (a,e)$.

Let $d = 0$. Then $(d,c) = (0,c) \eqo 0$. From $(b,c) \eqo e$, it follows by ($\star$) that $e \leq b$. Hence $(a,e) \leqcomp (a,b) \eqo 0 \eqo (0,0) \leqcomp (a,e)$ and, by (P1), $(a,e) \eqo 0$. In particular, $(a,e) \eqo (d,c)$.
\end{proof}

We finally see how Rees quotients are formed in our framework.

\begin{proposition} \label{prop:rees-quotient-by-level-sets}

Let $(\So^2 \sep \leqcomp, \eqo, \oneone)$ be a negative tomonoid partition and let $q \in \So$. Let $\So_q = \{ a \in \So \colon a > q \} \disjointunion \{0\}$, where $0$ is a new element, and endow $\So_q$ with the total order extending the total order on $\{ a \in \So \colon a > q \}$ such that $0$ is the bottom element. Then, for each $c \in \withoutzero{{\So_q}}$, the $\eqo$-class of $c$ is contained in $(\withoutzero{{\So_q}})^2$. Let $\eqo_q$ be the equivalence relation on ${\So_q}^2$ whose classes are the $\eqo$-classes of each $c \in \withoutzero{{\So_q}}$ as well as the subset of ${\So_q}^2$ containing the remaining elements. Then $({\So_q}^2 \sep \leqcomp, \eqo_q, \oneone)$ is the Rees quotient of $\So^2$ by $q$.

\end{proposition}

\begin{proof}

Let $(\So \sep \leq, \mo, 1)$ be the corresponding negative tomonoid. Let $\mo_q$ be the binary operation on $\So_q$ such that $(\So_q \sep \leq, \mo_q, 1)$ is (under the obvious identifications) the Rees quotient of $\So$ by $q$. Let $({\So_q}^2 \sep \leqcomp, \eqo'_q, \oneone)$ be the associated tomonoid partition.

Let $a, b, c \in \So$ such that $c > q$ and $(a,b) \eqo c$. Then $a, b \geq c$ by Lemma \ref{lem:properties-of-tomonoid-partition}(i) and consequently $a, b > q$. We conclude that the $\eqo$-class of each $c \in \withoutzero{{\So_q}}$ is contained in $(\withoutzero{{\So_q}})^2$.

We have to show $\eqo'_q = \eqo_q$. Let $a, b, c \in S_q$ such that $c \neq 0$. Then $(a,b) \eqo'_q c$ iff $a \mo_q b = c$ iff $a \mo b = c$ iff $(a,b) \eqo c$. Hence the $\eqo'_q$-class of each $c \in \withoutzero{{\So_q}}$ coincides with the $\eqo$-class of $c$. There is only one further $\eqo'_q$-class, the $\eqo'_q$-class of $0$, which consequently consists of all elements of ${\So_q}^2$ not belonging to the $\eqo$-class of any $c \in \withoutzero{{\So_q}}$.
\end{proof}

We may geometrically interpret Proposition \ref{prop:rees-quotient-by-level-sets} as follows. The Rees quotient by an element $q$ arises from the partition on $\So^2$ by removing all columns and rows indexed by elements $\leq q$ and by adding instead a single new column from left and a single new row from below. Moreover, all elements of the new column and the new row as well as all remaining elements that originally belonged to a class of some $a \leq q$ are joined into a single class, which is the class of the new bottom element. The classes of elements strictly larger than $q$ remain unchanged.

In the special case that $q$ is the atom of a {\fntomonoid}, just the left-most two columns and the lowest two rows are merged in this way. Figure \ref{fig:Rees-chain} shows the chain obtained from a $9$-element tomonoid by applying this procedure repeatedly.

\begin{figure}[ht!]

\begin{center}

\figrees

\end{center}

\caption{Beginning with the 9-element tomonoid shown in Figure \ref{fig:tomonoid-partition}, the successive formation of Rees quotients by the atom leads finally to the trivial tomonoid.
\label{fig:Rees-chain}}

\end{figure}

\section{Rees coextensions: the Archimedean case}
\label{sec:Archimedean-extensions}

We now turn to the problem of determining all one-element coextensions of a finite, negative tomonoid. In this section, we will restrict to those tomonoids that fulfil the Archimedean property.

A negative tomonoid $\So$ is called {\it Archimedean} if, for any $a \leq b < 1$, there is an $n \geq 1$ such that $b^n \leq a$. Here, $b^n = b \mo \ldots \mo b$ ($n$ factors). Note that, in the finite case, Archimedeanicity is obviously equivalent to nilpotency. Indeed, a {\fntomonoid}, whose bottom element is $0$, is Archimedean if and only if there is an $n \geq 1$ such that $a^n = 0$ for all $a < 1$. Note furthermore that negative tomonoids with at most two elements are trivially Archimedean.

We begin by characterising the Archimedean f.n.\ tomonoid partitions.

\begin{lemma} \label{lem:Archimedean-tomonoid}

Let $(\So^2 \sep \leqcomp, \eqo, \oneone)$ be a f.n.\ tomonoid partition. The following statements are pairwise equivalent:
\begin{itemize}

\item[\rm (i)] $\So^2$ is Archimedean.

\item[\rm (ii)] $(b,a) \noteqo (1,a)$ for any $a \in \withoutzero{\So}$ and $b < 1$.

\item[\rm (iii)] $(a,b) \noteqo (a,1)$ for any $a \in \withoutzero{\So}$ and $b < 1$.

\end{itemize}

\end{lemma}

\begin{proof}

Let $(\So \sep \leq, \mo, 1)$ be the corresponding {\fntomonoid} and let $\zeroo$ be the bottom element of $\So$. W.l.o.g., we can assume $0 \neq 1$. We show that (i) and (ii) are equivalent. The equivalence of (i) and (iii) is seen similarly.

Assume that (ii) holds. By the negativity of $\So$, we have $b \mo a < a$ for all $a \neq \zeroo$ and $b < 1$. Let $a < 1$. Then, for any $n \geq 1$, either $a^{n+1} < a^n$ or $a^n = 0$. As $\So$ is finite, the latter possibility applies for a sufficiently large $n$. It follows that $\So$ is Archimedean.

Assume that (ii) does not hold. Let $a \neq \zeroo$ and $b < 1$ such that $b \mo a = a$. As $\So$ is negative, we then have $a \leq b$ and it follows $b^n \geq b^{n-1} \mo a = a > 0$ for any $n \geq 2$. Hence $\So$ cannot be Archimedean.
\end{proof}

We shall construct coextensions of Archimedean {\fntomonoid}s that are Archimedean again. Let us outline our procedure, adopting an intuitive point of view.

To begin with, we again identify the tomonoid partition with a partitioned square array; cf.\ Figure \ref{fig:tomonoid-partition}. We enlarge this square, doubling the lowest row and left-most column. The equivalence relation $\eqe$ making the enlarged square into a tomonoid partition will then be constructed in two steps. First, we determine what we call the ramification, which is based on an equivalence relation $\eqp$ contained in the level equivalence of any Archimedean one-element coextension. Second, we apply a simple procedure to choose the final equivalence relation $\eqe$. To this end, certain $\eqp$-classes have to be merged such that the part of the square containing the classes of the new tomonoid's bottom element and atom is divided up into exactly two $\eqe$-classes.

For a chain $(\So \sep \leq)$, we denote by $(\Se \sep \leq)$ its {\it zero doubling extension}: we put $\Se = \withoutzero{\So} \disjointunion \{ \zeroe, \alpha \}$, where $\zeroe, \alpha$ are new elements, and we endow $\Se$ with the total order extending the total order on $\withoutzero{\So}$ such that $\zeroe < \alpha < a$ for all $a \in \withoutzero{\So}$. Furthermore, let $(\So \sep \leq, \mo, 1)$ be a \fntomonoid. Then we assume any one-element coextension of $\So$ to be of the form $(\Se \sep \leq, \me, 1)$. In particular, the intersection of $\So$ and $\Se$ is exactly $\withoutzero{\So}$ and $a \me b = a \mo b = c$ whenever $a, b, c \in \withoutzero{\So}$.

\begin{definition} \label{def:intermediate-equivalence-relation-archimedean}

Let $(\So^2 \sep \leqcomp, \eqo, \oneone)$ be an Archimedean {\fntomonoid} partition. Let $\Se = \withoutzero{\So} \disjointunion \{ \zeroe, \alpha \}$ be the zero doubling extension of $\So$. We define
\begin{equation} \label{fml:support}
\begin{split}
& \support \;=\; \{ (a,b) \in \Se^2 \colon 
           \text{$a, b \in \withoutzero{\So}$
                and there is a $c \in \withoutzero{\So}$ such that $(a,b) \eqo c$} \}, \\
& \cosupport \;=\; \Se^2 \setminus \support.
\end{split}
\end{equation}
Let $\eqp$ be the smallest equivalence relation on $\Se^2$ such that the following conditions hold:

\begin{itemize}

\axiom{E1} For any $(a, b), (c, d) \in \support$ such that $(a, b) \eqo (c, d)$, we have $(a, b) \eqp (c, d)$.

\axiom{E2} For any $(a,b), (b,c) \in \support$ and $d, e \in \withoutzero{\So}$ such that $(d,c), (a,e) \in \cosupport$, $\;(a,b) \eqo d$, and $(b,c) \eqo e$, we have $(d,c) \eqp (a,e)$.

\axiom{E3} For any $a, b, c, e \in \withoutzero{\So}$ such that $(a,b) \in \cosupport$, $\;(b,c) \eqo e$, and $c < 1$, we have $(a,e) \eqp \zeroe$.

Moreover, for any $a, b, c, d \in \withoutzero{\So}$ such that $(b,c) \in \cosupport$, $\;(a,b) \eqo d$, and $a < 1$, we have $(d,c) \eqp \zeroe$.

\axiom{E4} We have $(\zeroe, 1) \eqp (1, \zeroe) \eqp (\atome, b) \eqp (b, \atome)$ for any $b < 1$, and $(\atome, 1) \eqp (1, \atome)$. Moreover, for any $(a,b), (c,d) \in \cosupport$ such that $(a,b) \leqcomp (c,d) \eqp \zeroe$, we have $(a,b) \eqp \zeroe$.

\end{itemize}

Then we call the structure $(\Se^2 \sep \leqcomp, \eqp, \oneone)$ the {\it $(1,1)$-ramification} of $(\So^2 \sep \leqcomp, \eqo, \oneone)$.

\end{definition}

In this section, we will refer to the $(1,1)$-ramification also simply as the ``ramification''. The reason of the reference to the pair $(1,1)$ will become clear only in the next section.

Let us review Definition \ref{def:intermediate-equivalence-relation-archimedean} in order to elucidate how the ramification $(\Se^2 \sep \leqcomp, \eqp,$ $\oneone)$ arises from a tomonoid partition $(\So^2 \sep \leqcomp, \eqo, \oneone)$. We note first that $\support$ consists of all pairs $(a, b) \in \So^2$ whose product in $\So$ is not the bottom element. Indeed, by the negativity of $\So$, $(a, b) \eqo c$ and $c \in \withoutzero{\So}$ implies $a, b \in \withoutzero{\So}$. In other words, $\support$ is the union of the $\eqo$-classes of all $c \in \withoutzero{\So}$ and this union lies in ${\withoutzero{\So}}^2$. We furthermore note that $\support$ is an upwards closed subset of $\Se^2$. This is a consequence of the regularity of $\eqo$; cf.\ condition (P1'') in Lemma \ref{lem:regularity}. Consequently, $\cosupport$ is a downward closed subset of $\Se^2$.

Inspecting the conditions (E1)--(E4) defining $\eqp$, we next observe that $\eqp$-equivalences of elements of $\support$ are exclusively required by condition (E1). From this fact we conclude that the $\eqo$-class of any $c \in \withoutzero{\So}$ is also a $\eqp$-class. Hence the $\eqo$-classes contained in $\support$ are $\eqp$-classes as well. Note that this also means that the sets $\support$ and $\cosupport$ are uniquely determined by the ramification. In fact, $\support$ contains the $\eqp$-classes of all $c \in \withoutzero{\So}$ and $\cosupport$ contains all remaining $\eqp$-classes.

The $\eqp$-classes contained in $\cosupport$ are in turn defined by conditions (E2)--(E4).  See Figure \ref{fig:E2-E3} for an illustration of conditions (E2) and (E3). Note that each prescription contained in (E2) and (E3) is of the form that certain $\eqo$-equivalences imply that a certain pair of elements of $\cosupport$ is $\eqp$-equivalent. Finally, (E4) prescribes that the $\eqp$-class of $\zeroe$ is downward closed. We remark that $\cosupport$ contains the $\eqp$-classes of the bottom element $\zeroe$ and the atom $\alpha$, but possibly further $\eqp$-classes, which do not contain $(1,c)$ or $(c,1)$ for any $c \in \Se$.

\begin{figure}[ht!]
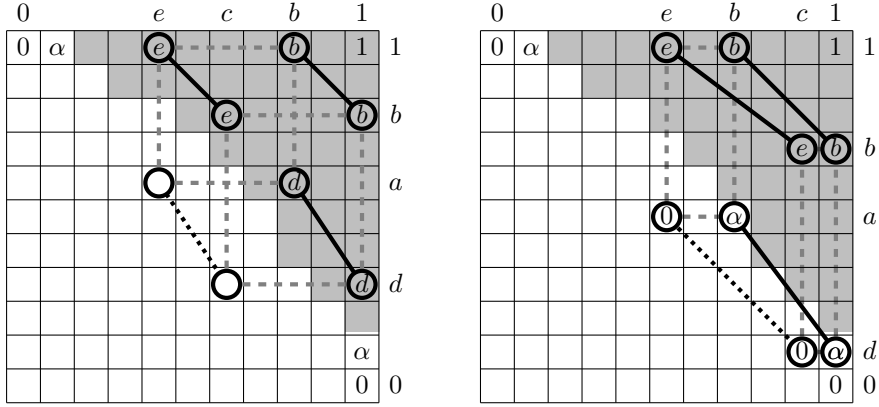


\begin{center}

\figarch

\end{center}

\caption{The prescriptions (E2) and (E3) in the construction of Rees coextensions for the Archimedean case.} \label{fig:E2-E3}

\end{figure}

For two equivalence relations $\eqo_1$ and $\eqo_2$ on a set $A$, we say that $\eqo_2$ is a {\it coarsening} of $\eqo_1$ if $\eqo_1 \subseteq \eqo_2$, that is, if each $\eqo_2$-class is a union of $\eqo_1$-classes.

\begin{lemma} \label{lem:intermediate-equivalence-relation-archimedean}

Let $(\So^2 \sep \leqcomp, \eqo, \oneone)$ be an Archimedean {\fntomonoid} partition and let $(\Se^2 \sep \leqcomp, \eqe, \oneone)$ be an Archimedean one-element coextension of $\So^2$. Furthermore, let $(\Se^2 \sep$ $\leqcomp, \eqp, \oneone)$ be the $(1,1)$-ramification of $\So^2$. Then $\eqe$ is a coarsening of $\eqp$ such that the following holds: the $\eqe$-class of each $c \in \withoutzero{\So}$ coincides with the $\eqp$-class of $c$, the $\eqe$-class of $\zeroe$ is downward closed, and each $\eqe$-class contains exactly one element of the form $(1,c)$ for some $c \in \Se$.

\end{lemma}

\begin{proof}
Let $(\So \sep \leq, \mo, 1)$ and $(\Se \sep \leq, \me, 1)$, where $\Se = \withoutzero{\So} \disjointunion \{ \zeroe, \alpha \}$, be the two tomonoids in question.

As noted above, condition (E1) requires $\eqp$-equivalences only between elements of $\support$ and the remaining conditions require $\eqp$-equivalences only between elements of $\cosupport$. Furthermore, $\support$ is the union of the $\eqo$-classes of all $c \in \withoutzero{\So}$. By (E1), these $\eqo$-classes are also $\eqp$-classes. Moreover, by Proposition \ref{prop:rees-quotient-by-level-sets}, each $\eqo$-class of a $c \in \withoutzero{\So}$ is a $\eqe$-class. We conclude that the $\eqe$-class of each $c \in \withoutzero{\So}$ coincides with the $\eqp$-class of $c$ and $\support$ is the union of these subsets.

We next check that any two elements that are $\eqp$-equivalent according to one of the conditions (E2)--(E4) are also $\eqe$-equivalent. Since $\eqp$ is, by assumption, the smallest equivalence relation with the indicated properties, it will then follow that $\eqp \subseteq \eqe$.

Ad (E2): Let $(a,b), (b,c) \in \support$, $\;d, e \in \withoutzero{\So}$, $\;(a,b) \eqo d$, and $(b,c) \eqo e$. Then $a, b, c \in \withoutzero{\So}$, hence $a \me b = a \mo b = d$ and $b \me c = b \mo c = e$. Consequently, $d \me c = (a \me b) \me c = a \me (b \me c) = a \me e$, that is $(d,c) \eqe (a,e)$.

Ad (E3): Let $a, b, c, e \in \withoutzero{\So}$, $\;(a,b) \in \cosupport$, $\;(b,c) \eqo e$, and $c < 1$. Then $a \me b \leq \atome$ and hence $a \me e = a \me (b \me c) = (a \me b) \me c \leq \atome \me c$. As $\So$ is assumed to be Archimedean, $\atome$ is the atom of $\Se$, and $c < 1$, we conclude $\atome \me c = \zeroe$. Hence $(a,e) \eqe 0$. Similarly, we argue for the second part of (E3).

Ad (E4): As $\So$ is Archimedean, we have, for any $b < 1$, $\;0 \me 1 = 1 \me \zeroe = \atome \me b = b \me \atom = \zeroe$ by Lemma \ref{lem:Archimedean-tomonoid} and hence $(\zeroe, 1) \eqe (1, \zeroe) \eqe (\atome, b) \eqe (b, \atome)$. Furthermore, we have $(\atome, 1) \eqe (1, \atome)$. Finally, let $(a,b), (c,d) \in \cosupport$ and assume $(a,b) \leqcomp (c,d) \eqe \zeroe$. Then $a \me b \leq c \me d = \zeroe$ and thus $(a,b) \eqe \zeroe$ as well.

It is finally clear that the $\eqe$-class of $\zeroe$ is downward closed. The last statement holds by condition (P2) of a tomonoid partition.
\end{proof}

We are now ready to state the main result of this section.

\begin{theorem} \label{thm:construction-of-Archimedean-elementary-extensions}

Let $(\So^2 \sep \leqcomp, \eqo, \oneone)$ be an Archimedean {\fntomonoid} partition and let $(\Se^2 \sep \leqcomp, \eqp, \oneone)$ be the $(1,1)$-ramification of $\So^2$. Let $\eqe$ be a coarsening of $\eqp$ such that the following holds: the $\eqe$-class of each $c \in \withoutzero{\So}$ coincides with the $\eqp$-class of $c$, the $\eqe$-class of $\zeroe$ is downward closed, and each $\eqe$-class contains exactly one element of the form $(1,c)$ for some $c \in \Se$. Then $(\Se^2 \sep \leqcomp, \eqe, \oneone)$ is an Archimedean one-element coextension of $\So^2$.

Moreover, all Archimedean one-element coextensions of $\So^2$ arise in this way.

\end{theorem}

\begin{proof}
$\support$, defined by (\ref{fml:support}), is the union of the $\eqo$-classes of all $c \in \withoutzero{\So}$. As we have seen in the proof of Lemma \ref{lem:intermediate-equivalence-relation-archimedean}, these subsets of $\support$ are also $\eqp$-classes. Recall also that $\support$ is upwards closed and $\cosupport = \Se^2 \setminus \support$ is downward closed.

By (E4), we have $(1, \zeroe) \eqp (\zeroe, 1)$ and $(1, \atome) \eqp (\atome, 1)$. We claim that $(1, \zeroe) \noteqp (1, \atome)$. Indeed, (E1), (E2), and (E3) involve only elements $(a,b)$ such that $a, b \in \withoutzero{\So}$. Hence, none of these prescriptions involves the elements $(1, \atome)$ or $(\atome, 1)$. Moreover, by (E4), the elements $(a, \zeroe)$ and $(\zeroe, a)$ for any $a$ as well as $(a, \atome)$ and $(\atome, a)$ for any $a \neq 1$ belong to the $\eqp$-class of $(1,\zeroe)$. Again, $(1, \atome)$ and $(\atome, 1)$ are not concerned. Finally, the $\eqp$-class of $(1,\zeroe)$ is a downward closed set. Also this prescription has no effect on $(1, \atome)$ or $(\atome, 1)$ because there is no element in $\cosupport$ that is larger than $(1, \atome)$ or $(\atome, 1)$. We conclude that $\{ (1, \atome), (\atome, 1) \}$ is an own $\eqp$-class and our claim is shown.

Let now $\eqe \supseteq \eqp$ be as indicated. Note that, by what we have seen so far, at least one such equivalence relation exists. In accordance with Proposition \ref{prop:PartAx-ass-simplified}, we will verify (P1''), (P2), and (P3').

We have shown that $(1,c) \eqe (c,1)$ for all $c \in \Se$. By construction, $\eqe$ fulfils (P2). Furthermore, the $\eqe$-class of $\zeroe$ is downward closed and $\cosupport$, which is the union of the $\eqe$-classes of $\zeroe$ and $\atome$, is downward closed as well. We conclude that (P1'') holds for $\eqe$.

It remains to show that $\eqe$ fulfils (P3'). Let $a, b, c, d, e \in \So \setminus \{ \zeroe, 1 \}$ such that $(a,b) \eqe d$ and $(b,c) \eqe e$. We distinguish the following cases.

{\it Case 1.} Let $d, e \in \withoutzero{\So}$. Then $(a,b) \eqo d$ and $(b,c) \eqo e$. As $\eqo$ fulfils (P3), we have $(d,c) \eqo (a,e)$. In particular, it follows that $(d,c) \in \support$ iff $(a,e) \in \support$. If $(d,c)$ and $(a,e)$ are both in $\support$, we have $(d,c) \eqe (a,e)$ because the $\eqp$-classes contained in $\support$ are $\eqe$-classes as well. If $(d,c)$ and $(a,e)$ are both in $\cosupport$, we have $(d,c) \eqp (a,e)$ by (E2) and consequently also $(d,c) \eqe (a,e)$, because $\eqe$ extends $\eqp$.

{\it Case 2.} Let $d = \atome$ and $e \in \withoutzero{\So}$. Then $(d,c) \eqp \zeroe$ by (E4). Furthermore, we have $a \in \withoutzero{\So}$ by (E4), $b, c \in \withoutzero{\So}$ because $(b,c) \in \support$, $(a,b) \in \cosupport$, and $(b,c) \eqo e$. It follows $(a,e) \eqp \zeroe$ by (E3). Consequently, $(d,c) \eqe \zeroe \eqe (a,e)$.

{\it Case 3.} Let $d \in \withoutzero{\So}$ and $e = \atome$. We argue similarly to Case 2.

{\it Case 4.} Let $d = e = \atome$. Then $(d,c) \eqp (a,e) \eqp \zeroe$ by (E4) and consequently also $(d,c) \eqe (a,e)$.

By Proposition \ref{prop:PartAx-ass-simplified}, $(\Se^2 \sep \leqcomp, \eqe, \oneone)$ is a {\fntomonoid} partition, which is moreover Archimedean by (E4) and Lemma \ref{lem:Archimedean-tomonoid}. It is finally clear from Proposition \ref{prop:rees-quotient-by-level-sets} that the Rees quotient of $\Se^2$ by the atom $\atome$ is $\So^2$.

The final statement follows from Lemma \ref{lem:intermediate-equivalence-relation-archimedean}.
\end{proof}

Let us exhibit some features of our construction. Starting from a tomonoid partition $(\So^2 \sep \leqcomp, \eqo, \oneone)$, we determine its $(1,1)$-ramification $(\Se^2 \sep \leqcomp, \eqp, \oneone)$ by applying conditions (E1)--(E4) from Definition \ref{def:intermediate-equivalence-relation-archimedean}. These prescriptions are largely independent. This is to say that, in order to determine $\eqp$, it is not necessary to use already obtained results in a recursive way. Furthermore, to obtain a coextension of the desired type, the set $\mathcal Z = \eqecl{(1,\zeroe)}$, i.e.\ the $\eqe$-class of the bottom element, must be chosen. Theorem \ref{thm:construction-of-Archimedean-elementary-extensions} characterises $\mathcal Z$ as follows: $\mathcal Z$ is a union of $\eqp$-classes contained in $\cosupport$ including $\eqpcl{(1,\zeroe)}$ but excluding the $\eqp$-class $\{(1, \atome), (\atome, 1)\}$, and $\mathcal Z$ is downward closed. Thus, to determine a specific one-element coextension, all we have to do is to select an arbitrary set of $\eqp$-classes different from $\{(\atome, 1), (1, \atome)\}$ and $\mathcal Z$ will then be the smallest downward closed set containing them.

We may in particular mention a simple fact: the explained procedure of determining an Archimedean extension always leads to a result. That is, every Archimedean, finite, negative tomonoid has at least one Archimedean one-element coextension. Indeed, with respect to the above notation, we may always choose $\mathcal Z = \cosupport \setminus \{(\atome, 1), (1, \atome)\}$. In general, it might be found interesting that the explained procedure never requires revisions. In fact, neither the determination of $\eqp$ nor of $\eqe$ involves decisions that lead to an impossible situation, we always end up with a structure of the desired type.

\begin{remark} \label{rem:extension}

We may characterise the set of Archimedean one-element coextensions of an Archimedean {\fntomonoid} $(\So \sep \leq, \mo, 1)$ also as follows.

Recall first that with any preorder $\preorder$ on a set $A$ we can associate a partial order, called its symmetrisation. Indeed, $\preorder$ gives rise to the equivalence relation $\preorderequ$, where $a \preorderequ b$ if $a \preorder b$ and $b \preorder a$, and on the quotient $\preorderequcl{A}$, $\preorder$ induces a partial order.

Referring to the notation of Lemma {\rm \ref{lem:intermediate-equivalence-relation-archimedean}}, let $\mathcal E$ be the set of all $\eqp$-classes contained in $\cosupport$. Then $\leqcomp$ induces on $\mathcal E$ the preorder $\preorderblockp$ (cf.\ Section {\rm \ref{sec:Level-sets}}). We can describe the extensions of $\So$ with exclusive reference to the preordered set $(\mathcal E \sep \preorderblockp)$. Namely, by Theorem {\rm \ref{thm:construction-of-Archimedean-elementary-extensions}}, there is one-to-one correspondence between the Archimedean one-element coextensions of $\So$ and the extensions of the preorder $\preorderblockp$ on $\mathcal E$ to a preorder whose symmetrisation consists of two elements, one of which contains $\eqpcl{(1,0)}$ and one of which contains $\eqpcl{(1,\atome)}$.

\end{remark}

\subsection*{The commutative case}

We conclude this section by considering the commutative case. Given a commutative, Archimedean {\fntomonoid} $\So$, our question is how to determine all its commutative, Archimedean one-element coextensions.

It is clear that we may apply to this end Theorem \ref{thm:construction-of-Archimedean-elementary-extensions}. Among the one-element coextensions of $\So$ we may simply select those that are commutative. An easy criterion of commutativity is stated in Lemma \ref{lem:properties-of-tomonoid-partition}(ii).

However, it would be desirable to apply a more direct procedure, with the effect that no result must be discarded. This turns out to be easy. All we have to do is to adapt the notion of a $(1,1)$-ramification. We add in Definition \ref{def:intermediate-equivalence-relation-archimedean} the following condition:
\begin{itemize}
\axiom{E5} For any $a, b \in \Se$ such that $(a,b), (b,a) \in \cosupport$, we have $(a,b) \eqp (b,a)$.
\end{itemize}
On the basis of this modified notion of a $(1,1)$-ramification, we may reformulate Theorem \ref{thm:construction-of-Archimedean-elementary-extensions} in order to deal with the commutative case only. We omit the straightforward details.

\section{Rees coextensions: the general case} \label{sec:Extensions}

We now turn to the construction of one-element coextensions of finite, negative tomonoids without any further restriction. The procedure is slightly more involved than in the Archimedean case.

To see what makes the difference, recall that, by Lemma \ref{lem:Archimedean-tomonoid}, a characteristic feature of the procedure explained in Theorem \ref{thm:construction-of-Archimedean-elementary-extensions} was the following: both the column and the row indexed by the atom $\atome$ contain, with the exception of $(1,\atome)$ and $(\atome,1)$, solely elements in the class of $\zeroe$. In the general case, this line and row may contain further elements of the class of $\atome$. For our general construction, we have to make a decision about the division of this line and row into members of the classes of $\zeroe$ and $\atome$.

An element $\idp$ of a tomonoid is called \emph{idempotent} if $\idp \mo \idp = \idp$.

\begin{lemma} \label{lem:atom-idempotent}

Let $(\So \sep \leq, \mo, 1)$ be a non-trivial {\fntomonoid}. Let $0, \atom$ be its bottom element and its atom, respectively. Then there is an idempotent $\leftidp \geq \atom$ in $\So$ such that
\[ a \mo \atome \;=\;
\begin{cases} \zeroo & \text{if $a < \leftidp$,} \\
              \atome & \text{if $a \geq \leftidp$.} \\
\end{cases} \]
Similarly, there is an idempotent $\rightidp \geq \atome$ in $\So$ such that
\[ \atome \mo a \;=\;
\begin{cases} \zeroo & \text{if $a < \rightidp$,} \\
              \atome & \text{if $a \geq \rightidp$.} \\
\end{cases} \]
\end{lemma}

\begin{proof}

We have $\zeroo \mo \atom = \zeroo$ and $1 \mo \atom = \atom$. Let $\leftidp \in \So$ be the smallest element $a \in \So$ such that $a \mo \atom = \atom$. Evidently, $\leftidp$ is not the bottom element. Furthermore, $\leftidp \mo \leftidp \mo \atom = \atom$ and, by the minimality of $\leftidp$, it follows $\leftidp \mo \leftidp = \leftidp$, that is, $\leftidp$ is an idempotent. The second part is proved analogously.
\end{proof}

Given a tomonoid $\So$, let us call the pair $(\leftidp, \rightidp)$ of idempotents, as specified in Lemma \ref{lem:atom-idempotent}, \emph{atom-characterising}. The notion is to indicate that these two elements uniquely determine the inner left and right translation associated with the atom of $\So$.

For the construction of a one-element coextension $\Se$ of a {\fntomonoid} $\So$, we will fix the atom-characterising idempotents in advance. Note that we can identify each non-zero idempotent $e$ of $\Se$ with an idempotent element of $\So$. In fact, either $e \in \withoutzero{\So}$ and hence $e$ is a non-zero idempotent of $\So$, or $e$ is the atom of $\Se$, in which case we can identify $e$ with the bottom element of $\So$. We say that $\Se$ is a one-element coextension of $\So$ \emph{with respect to $(\leftidp, \rightidp)$} if $(\leftidp, \rightidp)$ is a pair of idempotents of $\So$ and, under the above identification, the atom-characterising pair of idempotents of $\Se$.

\begin{definition} \label{def:intermediate-equivalence-relation}

Let $(\So^2 \sep \leqcomp, \eqo, \oneone)$ be a {\fntomonoid} partition. Let $\Se = \withoutzero{\So} \disjointunion \{ \zeroe, \alpha \}$ be the zero doubling extension of $\So$. Define $\support, \cosupport \subseteq \Se^2$ according to (\ref{fml:support}).

Moreover, let $(\leftidp, \rightidp)$ be a pair of idempotents of $\So$. Let $\eqp$ be the smallest equivalence relation on $\Se^2$ such that {\rm (E1)}, {\rm (E2)}, as well as the following conditions hold:

\begin{itemize}
\axiom{E3'} \begin{itemize}
\item[\rm (a)] For any $a, b, c, e \in \withoutzero{\So}$ such that $(a,b) \in \cosupport$, $\;(b,c) \eqo e$, and $c < \rightidp$, we have $(a,e) \eqp \zeroe$.

Moreover, for any $a, b, c, d \in \withoutzero{\So}$ such that $(b,c) \in \cosupport$, $\;(a,b) \eqo d$, and $a < \leftidp$, we have $(d,c) \eqp \zeroe$.

\item[\rm (b)] For any $a, b, c, e \in \withoutzero{\So}$ such that $(a,b) \in \cosupport$, $\;(b,c) \eqo e$, and $c \geq \rightidp$, we have $(a,e) \eqp (a,b)$.

Moreover, for any $a, b, c, d \in \withoutzero{\So}$ such that $(b,c) \in \cosupport$, $\;(a,b) \eqo d$, and $a \geq \leftidp$, we have $(d,c) \eqp (b,c)$.

\item[\rm (c)] For any $a, b, c > \zeroe$ such that $(a,b), (b,c) \in \cosupport$, $\;a < \leftidp$, and $c \geq \rightidp$ we have $(a,b) \eqp \zeroe$.

Moreover, for any $a, b, c > \zeroe$ such that $(a,b), (b,c) \in \cosupport$, $\;a \geq \leftidp$, and $c < \rightidp$ we have $(b,c) \eqp \zeroe$.
\end{itemize}

\axiom{E4'} \begin{itemize}
\item[\rm (a)] We have $(1, \zeroe) \eqp (\zeroe, 1) \eqp (a,\atome) \eqp (\atome,b)$ for any $a < \leftidp$ and $b < \rightidp$. Moreover, for any $(a,b), (c,d) \in \cosupport$ such that $(a,b) \leqcomp (c,d) \eqp \zeroe$, we have $(a,b) \eqp \zeroe$.

\item[\rm (b)] We have $(1, \atome) \eqp (\atome, 1) \eqp (\leftidp, \atome) \eqp (\atome, \rightidp)$. Moreover, for any $(a,b), (c,d) \in \cosupport$ such that $(a,b) \geqcomp (c,d) \eqp \atome$, we have $(a,b) \eqp \atome$.
\end{itemize}

\end{itemize}

Then we call the structure $(\Se^2 \sep \leqcomp, \eqp, \oneone)$ the {\it $(\leftidp, \rightidp)$-ramification} of $(\So^2 \sep \leqcomp, \eqo,$ $\oneone)$.

\end{definition}

We see that Definition \ref{def:intermediate-equivalence-relation} is largely analogous to Definition \ref{def:intermediate-equivalence-relation-archimedean}. See Figure \ref{fig:E3prime} for an illustration of conditions (E3')(a) and (b).

\begin{figure}[ht!]
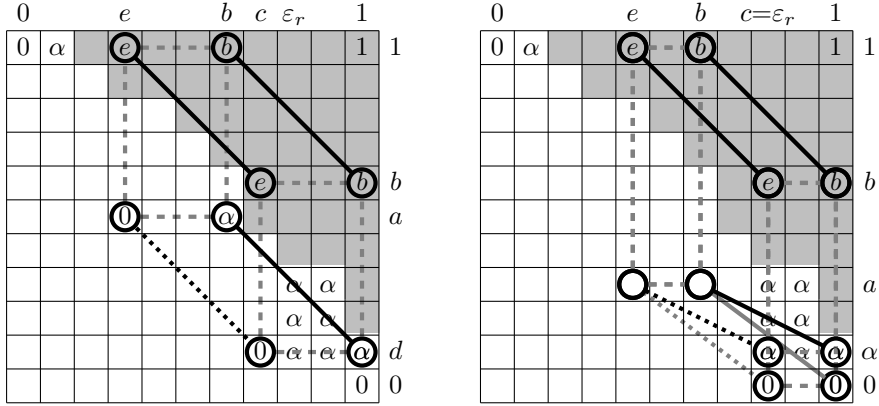


\begin{center}

\fignonarch

\end{center}

\caption{The prescriptions (E3')(a), (b) in the construction of one-element coextensions for the general case.} \label{fig:E3prime}

\end{figure}

Note that Definition \ref{def:intermediate-equivalence-relation-archimedean} is contained as a special case in Definition \ref{def:intermediate-equivalence-relation}. To see that both definitions are consistent, assume that $\So$ is an Archimedean {\fntomonoid}. It is immediate that in case $\leftidp = \rightidp = 1$ most conditions of Definitions \ref{def:intermediate-equivalence-relation} coincide with those of Definition \ref{def:intermediate-equivalence-relation-archimedean}. The only difference is that condition (E4')(b) does not possess an analog in (E4). In fact, (E4) does not require that the $\eqp$-class of the atom $\atom$ of $\Se$ is an upward closed subset of $\cosupport$. But we have seen in the proof of Theorem \ref{thm:construction-of-Archimedean-elementary-extensions} that the $\eqp$-class of $\atom$ is $\{ (1, \atom), (\atom, 1) \}$ and this set trivially fulfills the condition in question.

\begin{lemma} \label{lem:intermediate-equivalence-relation}

Let $(\So^2 \sep \leqcomp, \eqo, \oneone)$ be a {\fntomonoid} partition and let $(\Se^2 \sep \leqcomp, \eqe, \oneone)$ be a one-element coextension of $\So^2$ with respect to the idempotents $(\leftidp, \rightidp)$. Furthermore, let $(\Se^2 \sep \leqcomp, \eqp, \oneone)$ be the $(\leftidp, \rightidp)$\-/ramification of $\So$. Then $\eqe$ is a coarsening of $\eqp$ such that the following holds: the $\eqe$-class of each $c \in \withoutzero{\So}$ coincides with the $\eqp$-class of $c$, the $\eqe$-class of $\zeroe$ is downward closed, and each $\eqe$-class contains exactly one element of the form $(1,c)$ for some $c \in \Se$.

\end{lemma}

\begin{proof}

We denote again by $\mo$ and $\me$ the monoidal operations on $\So$ and $\Se = \withoutzero{\So} \disjointunion \{ \zeroe, \atome \}$, respectively.

We proceed similarly as in the proof of Lemma \ref{lem:intermediate-equivalence-relation-archimedean} to see that, for each $c \in \withoutzero{\So}$, the $\eqo$-class of $c$ coincides with the $\eqp$-class as well as with the $\eqe$-class of $c$, and $\support$ is the union of these sets.

To show that $\eqe$ is a coarsening of $\eqp$, we check that any $\eqp$-equivalence according to (E2), (E3'), or (E4') is an $\eqe$-equivalence as well. For (E2), see the proof of Lemma \ref{lem:intermediate-equivalence-relation-archimedean}. Furthermore, for (E4'), the argument is obvious from Lemma \ref{lem:atom-idempotent}. Thus we mention only the case of (E3'):

Ad (E3')(a): Let $a, b, c, e \in \withoutzero{\So}$, $(a,b) \in \cosupport$, $(b,c) \eqo e$, and $c < \rightidp$. Then $a \me b \leq \atome$ and $b \me c = e$ and hence $a \me e = a \me (b \me c) = (a \me b) \me c \leq \atome \me c = \zeroe$. Hence $(a,e) \eqe 0$. Similarly, we argue for the second part of (E3')(a).

Ad (E3')(b): Let $a, b, c, e \in \withoutzero{\So}$, $(a,b) \in \cosupport$, $(b,c) \eqo e$, and $c \geq \rightidp$. Then $a \me b = \zeroe$ or $a \me b = \atome$. In the former case we have $a \me e = a \me b \me c = \zeroe \me c = \zeroe$, and in the latter case we have $a \me e = a \me b \me c = \atome \me c = \atome$. We conclude $(a,e) \eqe (a,b)$. Similarly, we argue for the second part of (E3')(b).

Ad (E3')(c): Let $a, b, c \geq \atome$, $\;(a,b), (b,c) \in \cosupport$, $\;a < \leftidp$, and $c \geq \rightidp$. Assume that $a \me b = \atome$. Then $(a \me b) \me c = \atome \me c = \atome$, but $a \me (b \me c) \leq a \me \atome = \zeroe$. We conclude $a \me b = \zeroe$, that is, $(a,b) \eqe \zeroe$. Similarly, we argue for the second part of (E3')(c).

We complete the proof like in case of Lemma \ref{lem:intermediate-equivalence-relation-archimedean}.
\end{proof}

\begin{theorem} \label{thm:construction-of-elementary-extensions}

Let $(\So^2 \sep \leqcomp, \eqo, \oneone)$ be a {\fntomonoid} partition, let $(\leftidp, \rightidp)$ be a pair of idempotents of $\So$, and let $(\Se^2 \sep \leqcomp, \eqp, \oneone)$ be the $(\leftidp, \rightidp)$-ramification of $\So^2$.

If $(1,0) \eqp (1,\atome)$, there is no one-element coextension of $\So^2$ with respect to $(\leftidp, \rightidp)$.

Assume $(1,0) \noteqp (1,\atome)$. Let $\eqe$ be a coarsening of $\eqp$ such that the following holds: the $\eqe$-class of each $c \in \withoutzero{\So}$ coincides with the $\eqp$-class of $c$, the $\eqe$-class of $\zeroe$ is downward closed, and each $\eqe$-class contains exactly one element of the form $(1,c)$ for some $c \in \Se$. Then $(\Se^2 \sep \leqcomp, \eqe, \oneone)$ is a one-element coextension of $\So^2$ with respect to $(\leftidp, \rightidp)$.

Moreover, all one-element coextensions of $\So^2$ with respect to $(\leftidp, \rightidp)$, if there are any, arise in this way.

\end{theorem}

\begin{proof}
$\support$ is the union of the $\eqo$-classes of all $c \in \withoutzero{\So}$. We argue like in the previous cases that these subsets of $\support$ are also $\eqp$-classes. Note again that $\cosupport = \Se^2 \setminus \support$ is downward closed.

Assume that $(1, \zeroe) \eqp (1, \atome)$. If there was a one-element coextension $(\Se^2 \sep \leqcomp, \eqe, \oneone)$ of $\So^2$ with respect to $(\leftidp, \rightidp)$, we would have $\eqp \subseteq \eqe$ by Lemma \ref{lem:intermediate-equivalence-relation}. This means that $\eqe$ would violate (P2). Thus no such extension exists.

Assume now that $(1, \zeroe) \noteqp (1, \atome)$. Then $(1, \zeroe)$ and $(\zeroe, 1)$ on the one hand, and $(1,\atome)$, $(\atome,1)$, $(\leftidp, \atome)$, and $(\atome, \rightidp)$ on the other hand, are in distinct $\eqp$-classes. Obviously, an equivalence relation $\eqe \supseteq \eqp$ then exists as indicated. We readily see that $\eqe$ fulfils (P1'') and (P2).

We shall show that $\eqe$ fulfils also (P3'). Let $a, b, c, d, e \in \So \setminus \{ \zeroe, 1 \}$ such that $(a,b) \eqe d$ and $(b,c) \eqe e$. We have to show $(d,c) \eqe (a,e)$. To this end, we distinguish a number of cases and subcases.

{\it Case 1.} Let $d, e \in \withoutzero{\So}$. Then we proceed like in the proof of Theorem \ref{thm:construction-of-Archimedean-elementary-extensions} to conclude that $(d,c) \eqe (a,e)$.

{\it Case 2.} Let $d = \atome$ and $e \in \withoutzero{\So}$. Then $(a,b) \in \cosupport$ and $(b,c) \in \support$. Furthermore, we have $(b,c) \eqo e$ and $\atome < e \leq b, c$. We distinguish four subcases.

\mbox{}\hfill
\begin{minipage}{0.98\textwidth}\parindent0pt\parskip1ex
{\it Case a.} Let $a \in \withoutzero{\So}$ and $c < \rightidp$. Then $(d,c) = (\atome, c) \eqp \zeroe$ by (E4')(a). Furthermore, $(a,e) \eqp \zeroe$ by (E3')(a). It follows $(d,c) \eqe \zeroe \eqe (a,e)$.

{\it Case b.} Let $a = \atome$ and $c < \rightidp$. Then, again by (E4')(a), $(d,c) = (\atome, c) \eqp \zeroe$. Moreover, $(a,e) = (\atome,e) \eqp \zeroe$ by (E4')(a), because $e \leq c$. It follows $(d,c) \eqe \zeroe \eqe (a,e)$.

{\it Case c.} Let $a \in \withoutzero{\So}$ and $c \geq \rightidp$. Then $(d,c) = (\atome, c) \eqp \atome$ by (E4')(b) and thus $(d,c) \eqe \atome$. Furthermore, by (E3')(b), $(a,e) \eqp (a,b)$ and thus $(a,e) \eqe (a,b) \eqe d = \atome$. In particular, $(d,c) \eqe (a,e)$.

{\it Case d.} Let $a = \atome$ and $c \geq \rightidp$. Again, $(d,c) = (\atome, c) \eqe \atome$ by (E4')(b). Moreover, $(\atome,b) = (a,b) \eqe d = \atome$. It follows that $b \geq \rightidp$ because otherwise we would have $(\atome,b) \eqp \zeroe$ by (E4')(a) and thus $(\atome,b) \eqe \zeroe$. In $\So$, it follows from $\rightidp \leq b, c$ that $\rightidp = \rightidp \mo \rightidp \leq b \mo c = e$. Consequently, by (E4')(b), $(a,e) = (\atome,e) \eqp \atome$ and hence $(a,e) \eqe \atome$. We have shown $(a,e) \eqe (d,c)$.
\end{minipage}

{\it Case 3.} Let $d \in \withoutzero{\So}$ and $e = \atome$. Then we proceed analogously to Case 2.

{\it Case 4.} Let $d = e = \atome$. Then $(a,b), (b,c) \in \cosupport$. Assume that $a < \leftidp$ and $c \geq \rightidp$. Then by (E3')(c) it follows $(a, b) \eqp \zeroe$, in contradiction to $(a, b) \eqe \atome$. Similarly, we argue in case $a \geq \leftidp$ and $c < \rightidp$. We conclude that either $a < \leftidp$ and $c < \rightidp$, or $a \geq \leftidp$ and $c \geq \rightidp$. Thus $(a, e) = (a, \atome) \eqp (\atome, c) = (d, c)$ by (E4') and hence $(a, e) \eqe (d, c)$.

By Proposition \ref{prop:PartAx-ass-simplified}, $(\Se^2 \sep \leqcomp, \eqe, \oneone)$ is a {\fntomonoid} partition. Clearly, $\Se^2$ is associated with a one-element coextension of $\So^2$ with respect to $(\leftidp, \rightidp)$.

The final statement follows from Lemma \ref{lem:intermediate-equivalence-relation}.
\end{proof}

We conclude that the construction of one-element coextensions works, on the whole, for general {\fntomonoid}s similarly as for Archimedean {\fntomonoid}s. In fact, the explanations given after Theorem \ref{thm:construction-of-Archimedean-elementary-extensions} for the Archimedean case as well as Remark \ref{rem:extension} apply, mutatis mutandis, for the general case as well. Maybe one point is worth mentioning. In order to determine the downward closed set $\mathcal Z$, the $\eqe$-class of $\zeroo$, a set of $\eqp$-classes needs to be selected. In the Archimedean case, the $\eqp$-class $\{(\atome, 1), (1, \atome)\}$ must be disregarded. Here, in the general case, the $\eqp$-class {\it containing} $(\atome, 1)$ and $(1, \atome)$ must be disregarded instead.

The question remains if the pair of idempotents with respect to which we construct an one-element coextension can be chosen arbitrarily or not. The example shown in Figure \ref{fig:noncommutative-counterexample} implies that the answer is negative. It is an open problem how to characterise those pairs of idempotents that can be used.

However, in two cases a coextension always exists. On the one hand, there is always at least one coextension w.r.t.\ $(1,1)$. In fact, in this case we can argue like in the proof of Theorem \ref{thm:construction-of-Archimedean-elementary-extensions} to see that $\{(1,\atome), (\atome,1)\}$ is a $\eqp$-class and hence $(1,\zeroe) \noteqp (1,\atome)$. We can consequently choose, e.g., $\mathcal Z = \cosupport \setminus \{(1,\atome), (\atome,1)\}$. On the other hand, there is always exactly one extension w.r.t.\ $(\zeroo, \zeroo)$, where $\zeroo$ is the bottom element of $\So$. Then both atom-characterising idempotents of $\Se$ are $\atome$ and hence $\mathcal Z$ is necessarily the smallest possible set, namely, $\mathcal Z = \{ (a,b) \colon a = \zeroe \text{ or } b = \zeroe \}$. The new atom $\atom$ will be idempotent and the construction may be regarded as the ordinal sum of the original tomonoid and the two-element tomonoid whose monoidal product is the infimum.

\begin{figure}[ht!]

\begin{center}

\figcntex

\end{center}

\caption{Left: A 4-element tomonoid $\So$. The idempotents of $\So$ are $0$, $z$, and $1$. $\So$ does not possess a one-element coextension with respect to the pair of idempotents $(z,1)$. Middle: By (E4')(b), we have $(z,\atom) \eqp \atom$ and furthermore $(z,y) \eqp \atom$. Right: By (E4')(a) we have $(\atom,z) \eqp 0$ and by (E3')(a) we have $(z,y) \eqp 0$. We conclude that $(1,\zeroe) \eqp (1,\atome)$ in this case.
\label{fig:noncommutative-counterexample}}

\end{figure}

\subsection*{The commutative case}

Again, let us check which modifications of our procedure are necessary to deal with the commutative case.

Let $\Se$ be a commutative {\fntomonoid}. Then a one-element coextension of $\Se$ with respect to a pair $(\leftidp,\rightidp)$ of idempotents can obviously be commutative only if $\leftidp = \rightidp$. Consequently, we have to restrict Definition \ref{def:intermediate-equivalence-relation} to this case. In Definition \ref{def:intermediate-equivalence-relation}, we furthermore add again condition (E5). On the basis of these two modification, we may obtain an analogous version of Theorem \ref{thm:construction-of-elementary-extensions}, tailored to the commutative case.

\section{Conclusion} \label{sec:Conclusion}

In recent years, there has been a considerable interest in the structure of negative totally ordered monoids (tomonoids), which, for instance, occur in the context of fuzzy logic. Mostly, the commutative case has been studied. In particular, a classification of MTL-algebras is considered as an important aim. In the present paper, we focus on the finite case. Moreover, commutativity is not assumed. Our aim is to contribute to a better understanding of finite, negative (f.n.) tomonoids.

To this end, we have described the set of one-element Rees coextensions of a \fntomonoid, that is, the set of those {\fntomonoid}s whose Rees quotient by the atom is the original tomonoid. It has turned out convenient to employ in this context the level-set representation of tomonoids. We have thus worked with certain partitions of a set $\So^2$, where $\So$ is a finite chain. The construction consists of two steps and describes the extensions in a transparent, geometrically intuitive, and efficient way.

Among the open questions, we may mention the following. For an extension of a \fntomonoid, a pair of idempotents needs to be chosen in advance. Not all combinations, however, are possible and it is unclear how to characterise those pairs that are actually allowed.

Moreover, the tomonoids considered in this paper are finite and negative. Remarkably, finiteness is not an essential condition of our method. Alternatively, we could restrict to the assumption that the tomonoid has a bottom element. The extended tomonoid would then again consist of the non-zero elements of the original tomonoid together with a new pair of elements. The situation is more difficult, however, with regard to negativity. To generalise our method to the non-negative case would require major modifications. A related problem is the extension of our method to totally ordered semigroups, which do not necessarily possess an identity.

Finally, our description of a finite, negative tomonoid is relative to a tomonoid that is by one element smaller. For the sake of a classification of all finite, negative tomonoids it would certainly be desirable to understand the construction process not just step by step, but as a whole. This concern certainly implies the need for an approach going well beyond of what we have proposed in the present work.


\subsubsection*{Acknowledgements.} 
M.~P.\ was supported by Project P201/12/P055 of the Czech Science Foundation and the ESF Project CZ.1.07/2.3.00/20.0051 of the Masaryk University (Algebraic methods in Quantum Logic). Th.~V.\ was supported by the Austrian Science Fund (FWF): project I 1923-N25 (New perspectives on residuated posets).

We would moreover like to thank the anonymous reviewers for their constructive criticism, which led us to a considerable improvement of our work.


\end{document}